\documentclass{amsart}
\usepackage{amscd}
\usepackage{amsmath}
\usepackage{amssymb}
\usepackage{amsthm}
\begin{document}

\newtheorem{thm}{Theorem} 
\newtheorem*{thm*}{Theorem}
\newtheorem*{wish}{Wish}
\newtheorem{prop}{Proposition}[section] 
\newtheorem{llama}[prop]{Lemma} 
\newtheorem{sheep}[prop]{Corollary} 
\newtheorem*{sheep*}{Corollary} 
\newtheorem{deff}{Definition}
\newtheorem{lett}{Step}[section] 
\newtheorem{fact}[prop]{Fact}
\newtheorem*{fact*}{Fact}
\newtheorem{assume}{Assumption}[section]
\newtheorem{example}{Examples}[section]

\newcommand{\sthat}{\hspace{.1cm}| \hspace{.1cm}}
\newcommand{\p}{\mathbb{P}^1}
\newcommand{\ga}{\mathbb{G}_a}
\newcommand{\gm}{\mathbb{G}_m}
\newcommand{\spec}{ \mathit{Spec} }
\newcommand{\id}{\operatorname{id} }
\newcommand{\acl}{\operatorname{acl}}
\newcommand{\aclsig}{\operatorname{acl_\sigma}}

\newcommand{\tsig}{T_\sigma}
\newcommand{\lsig}{L_\sigma}
\newcommand{\sigdeg}{\deg_\sigma}
\newcommand{\dM}{\operatorname{dM}}
\newcommand{\RM}{\operatorname{RM}}
\newcommand{\zdense}{Zariski-dense }

\newcommand{\AAA}{\mathcal{A}}
\newcommand{\BB}{\mathcal{B}}
\newcommand{\CC}{\mathcal{C}}
\newcommand{\DD}{\mathcal{D}}
\newcommand{\EE}{\mathcal{E}}
\newcommand{\FF}{\mathcal{F}}
\newcommand{\MM}{\mathcal{M}}

\newcommand{\anandzoe}{Chatzidakis-Pillay }

\title{Grouplike minimal sets in ACFA and in $T_A$.}
\author{Alice Medvedev}

\address{Department of Mathematics, Statistics, and Computer Science \\
University of Illinois at Chicago \\
322 Science and Engineering Offices (M/C 249) \\
851 S. Morgan Street \\
Chicago, IL 60607-7045}
\email{alice@math.uic.edu}

\begin{abstract}
 This paper began as a generalization of a part of the author's PhD thesis about ACFA and ended up with a characterization of groups definable in $T_A$.
 The thesis concerns minimal formulae of the form $x \in A \wedge \sigma(x) = f(x)$ for an algebraic curve $A$ and a dominant rational function $f : A \rightarrow \sigma(A)$. These are shown to be uniform in the Zilber trichotomy, and the pairs $(A,f)$ that fall into each of the three cases are characterized. These characterizations are definable in families.
 This paper covers approximately half of the thesis, namely those parts of it which can be made purely model-theoretic by moving from ACFA, the model companion of the class of algebraically closed fields with an endomorphism, to $T_A$, the model companion of the class of models of an arbitrary totally-transcendental theory $T$ with an injective endomorphism, if this model-companion exists.
 A $T_A$ analog of the characterization of groups definable in ACFA is obtained in the process.
 The full characterization of the cases of the Zilber trichotomy in the thesis is obtained from these intermediate results with heavy use of algebraic geometry.
\end{abstract}

\maketitle

\section{Introduction}
\label{intro}

\subsection{Background and History and Motivation.}

  This paper is a generalization of a part of the author's PhD thesis \cite{mythesis}. A non-logician might find the exposition in the thesis more transparent, while the results in this paper are more general.
  In the thesis, we sought difference field analogs of the results that Hrushovski and Itai prove for differential fields in \cite{hruit}. They show that for some classes $C$ of strongly minimal sets definable in a differentially closed field, non-orthogonality to some type in $C$ is a definable property on families of definable sets. This allows them to produce a new model complete theory of differential fields (one for each such class $C$) realizing all types in differentially closed fields except for those non-orthogonal to something in $C$. The situation in difference fields is much more complicated, and we only make a first step toward this goal. However, our explicit description of the structure of certain definable sets in ACFA has proved tremendously useful for algebraic dynamics \cite{polydyn}.
  With those later results \cite{polydyn} we show that non-orthogonality between two minimal sets of the form $\sigma(x) = f(x)$ for a \emph{polynomial} $f$ in characteristic zero is often definable. We should say a little about difference fields before we give more details.

  A \emph{difference field} is a field $K$ with a distinguished endomorphism $\sigma$; it is naturally a structure for the language of rings augmented by a unary function symbol $\sigma$ denoting the endomorphism. This is a natural setting for studying functional equations, and it also turns out to be a useful formalism for studying algebraic dynamical systems, and for certain questions in arithmetic geometry.

 Functional equations like $f(x+1) = x f(x)$ that have been studied by analysts for several centuries fit into this formalism by taking $K$ to be, for example, the field of meromorphic functions on $\mathbb{C}$, and $\sigma(f(x)) = f(x+1)$. The name \emph{difference field} comes from considering an automorphism $\sigma(f(x)) = f(x+ \delta)$ for a fixed $\delta$ on some field of functions, and working with equations in finite difference quotients $(\frac{f(x+ \delta) - f(x)}{\delta})$ as an approximation to differential equations.
 Difference algebra, the study of difference polynomial rings and their ideals, was first developed by Ritt and Cohn and can be found in Cohn's book \cite{cohn}.
  There are several obstructions to developing difference-algebraic geometry as an analog to algebraic geometry. Cohn's book defines the correct analog of radical ideals, and proves that they satisfy the ascending chain condition. However, in contrast with plain algebraic geometry, it may be impossible to amalgamate two difference field extensions, which seems to preclude Weil-style universal domains. It is also unclear, from the algebraic point of view, how to define the dimension of a difference-closed set: for example, the natural Krull dimension fails to satisfy the fiber dimension theorem. Model theory offers solutions for both of these.

  Several people, including Macintyre and van den Dries, noticed in mid-90s that the class of difference fields has a model companion, called ACFA. Its models serve as universal domains for difference algebra; for that reason, they are sometimes called \emph{difference-closed fields}. ACFA has been studied extensively by Chatzidakis, Hrushovski, and others, especially in \cite{ChaHru1} and \cite{ChaHruPet}. In particular, Chatzidakis and Hrushovski show in \cite{ChaHru1} that ACFA is a supersimple theory, so the Lascar rank provides a good notion of dimension. In \cite{ChaHruPet}, it is shown that the minimal (Lascar rank 1) types in ACFA satisfy a version of the Zilber trichotomy: each is exactly one of fieldlike (nonorthogonal to a fixed field of a definable field automorphism), grouplike (nonorthogonal to a generic type of a minimal modular definable group), or trivial. They also show that the only definable field automorphisms are powers of $\sigma$, powers of the Frobenius automorphism, and compositions of these.

 Here is how model theory of difference fields, and specifically ACFA, is relevant to arithmetic.
 While each Frobenius endomorphism $\Phi^n (x) = x^{(p^n)}$ on a field of positive characteristic $p$ is already defined in the language of rings, one needs the language of difference fields (and the formalism of first-order logic) to speak of the ``limit theory'' of these structures, that is of the theory of a nonprincipal ultraproduct of them. Hrushovski shows in \cite{udifrob} that this limit theory is precisely ACFA.
  Hrushovski uses ACFA to give a new proof of the Manin-Mumford conjecture in \cite{udimm}, giving more explicit bounds for the number of torsion points of the Jacobian of an algebraic curve which lie on the curve. Indeed, the difference equations of the form $\sigma(x) = f(x)$ for a rational function $f$, the focus of the author's thesis \cite{mythesis}, figure prominently in his work.

  The author's thesis \cite{mythesis} concerns minimal formulae in models of ACFA of the form $x \in A \wedge \sigma(x) = f(x)$ for an algebraic curve $A$ and a dominant rational function $f : A \rightarrow \sigma(A)$. We prove that these formulae are uniform in the Zilber trichotomy, that is, that all non-algebraic types containing a given formula fall into the same case of the trichotomy; and we characterize the pairs $(A,f)$ that fall into each of the three cases. The fieldlike case is already characterized in \cite{ChaHru1} as purely inseparable (including linear) $f$; we show that $(A,f)$ gives a grouplike formula if and only if $f$ is not purely inseparable and either (1) $A$ is (birationally isomorphic to) an algebraic group curve (additive, multiplicative, or elliptic) and $f$ is (skew-conjugate to) a group homomorphism; or (2) $A$ is (birationally isomorphic to) $\mathbb{P}^1$ and $f$ is (skew-conjugate to) a generalized Latt\`{e}s function. Here, $f$ is skew-conjugate to $g$ if there is a birational isomorphism $L$ such that $g = L^\sigma \circ f \circ L^{-1}$. A generalized Latt\`{e}s function is a quotient of an isogeny of algebraic groups by a finite group of automorphisms of the algebraic group (see \cite{aboutlattes} for a beautiful exposition on Latt\`{e}s functions in characteristic $0$ and \cite{silverdyn} for more arithmetic dynamics). Precise definitions and details are given in the last section of this paper.

  This paper covers approximately half of the thesis, namely those parts of it which can be made purely model-theoretic by moving from ACFA, the model companion of the class of algebraically closed fields with an endomorphism, to $T_A$, the model companion of the class of models of an arbitrary totally-transcendental theory $T$ with an injective endomorphism, if this model-companion exists. See section\ref{meetta} below for details.
  With very recent work \cite{amador} by Blossier, Martin-Pizarro, and Wagner, we also obtain a characterization of definable groups in this generality. Although there is no hope for the Zilber Trichotomy in this generality, it is useful to generalize the results from ACFA to $T_A$ for three reasons. We are pleased to give a more model-theoretic account. We hope that this could be useful in other theories, most notably DCFA. Back in ACFA, this exposition clarifies which parts of the proof rely on $A$ being a curve (rather than a higher-dimensional variety), and which rely on $f$ being a single-valued function rather than a finite-valued correspondence; we hope this will help us eliminate these hypotheses.

 The author is most grateful to her thesis advisor, Thomas Scanlon; to Bjorn Poonen, who read the thesis and suggested several corrections and simplifications; and to Moshe Kamensky, whose interest in this generalization to $T_A$ inspired this paper, several drafts of which he read carefully. We thank the referee for the close reading and the detailed comments.

\subsection{Please meet $T_A$.}
\label{meetta}

 We take a stable theory $T$ which eliminates quantifiers and imaginaries in a language $L$. We denote $L$-definable sets by $\AAA$, $\BB$, etc. Sets denoted by non-italics might not be definable at all; occasionally they are definable in an expanded language $\lsig$, or type-definable in one of the languages.

\begin{deff}
 Let $\lsig := L \cup \{ \sigma\}$, where $\sigma$ is a new unary function symbol.\\
 Let  $T_\sigma := T \cup
  \{ \forall x\, \phi(x) \leftrightarrow \phi(\sigma(x)) \sthat \phi \in L \} \subset \lsig$.\\
  If $T_\sigma$ admits a model-companion, we call the model-companion $T_A$.\\
   In an $\lsig$-structure, we write $\aclsig$ for the usual model-theoretic algebraic closure, and $\acl$ for the algebraic closure in the reduct to $L$.
\end{deff}

 $T_\sigma$ asserts that $\sigma$ is an injective $L$-endomorphism. This theory is called $T_{Aut}$ in \cite{bash}, where Baldwin and Shelah give a necessary and sufficient condition for the existence of $T_A$, a strengthening of ``$T$ does not have the finite cover property'' that they call `` $T$ does not admit obstructions''. Chatzidakis and Pillay first examined the theory $T_\sigma$ in \cite{chapi} and gave sufficient conditions (3.11.2, p. 85) for the existence of its model-companion $T_A$ when $T$ is a theory of finite Morley rank. They prove the following:

\begin{fact} (\cite{chapi})
\label{basicfacts}
 If $T$ is stable and has quantifier elimination, and $T_A$ exists, then
\begin{enumerate}
\item $T_A$ is simple, and supersimple if $T$ is superstable. (Corollary 3.8, p. 84)
\item \label{alclosure} In a model of $T_A$,
   $\aclsig(A) = \acl( \cup_{i \in \mathbb{Z}} \sigma^i(A) )$. (Lemma 3.6, p. 82)
 \item $T_A$ eliminates quantifiers down to formulas of the form
   $\exists z\, \theta(x, z)$ where \begin{itemize}
    \item $z$ is a single variable,
    \item $\theta$ is quantifier-free, and
    \item $\tsig \vdash \forall x \exists_{<N} z\, \theta(x,z)$ for some $N \in \mathbb{N}$
   \end{itemize}
  (follows from Proposition 3.5(2) by the usual methods.)
 \item \label{indepinTA} Forking independence in $T_A$ is given by the following: $A$ and $B$ are independent over $E$ if and only if $\aclsig(EA)$ is independent \emph{in the sense of $T$} from  $\aclsig(EB)$ over $\aclsig(E)$. (Terminology and Theorem 3.7) \\
\end{enumerate} \end{fact}

It follows that forking in $T_A$ is always witnessed by quantifier-free formulae. The theory of algebraically closed fields satisfies the sufficient conditions for the existence of $T_A$ given in \cite{chapi} (3.11.2, p. 85), so the above fact applies to it. The model-companion ACFA of ACF$_\sigma$ is examined in great detail in \cite{ChaHru1, ChaHruPet}.

 It also follows from the characterization of $\aclsig$ and the characterization of independence that in the terminology of \cite{amador} ``$T_A$ is one-based over T with respect to $\aclsig$'', so their Theorem 2.15 applies, pinning down $\lsig$-definable groups in terms of $L$-definable ones.
 Although both theories are assumed stable in the statement of the theorem, the authors state on p.3
 that it is sufficient for the larger theory to be simple and the smaller to be stable, as is the case in this paper.

 \begin{fact}( Theorem 2.15 in \cite{amador})
\label{amadorfact}
In a model of $T_A$, for any $\lsig$-definable group $H$ there are a finite-index subgroup $H'$ of $H$, an $L$-definable group $A$, and an $\lsig$-definable group homomorphism $\phi: H' \rightarrow A$ with finite kernel.
\end{fact}

A better understanding of the proof of this fact may allow us to remove "totally transcendental" from the following assumption, but at present we do not see how to do that.

 \textbf{ We assume throughout this paper that $T$ is a totally transcendental 
theory with quantifier elimination and elimination of imaginaries, and that $T_A$ exists.}

\begin{deff}
 $\RM$ and $\dM$ are the Morley rank and degree of \emph{$L$-definable sets in the sense of $L$}. $U$ is the Lascar rank of $\lsig$-definable or $\lsig$-type-definable sets.
\end{deff}

 While the Lascar rank is normally defined for complete types, it is common to abuse notation and write $U(\pi)$ to mean $\sup \{ U(p) \sthat \pi \subset p \in S(A) \}$ when $\pi$ is a partial type over $A$.

\subsection{Outline of this paper}

 Sections \ref{moshethmsection}, \ref{groupsection}, and \ref{chase} are the heart of the paper. Here, unlike in the author's thesis, there is no algebraic geometry, and no appeal to any special properties of ACFA as opposed to $T_A$.

 The language used in this paper is the language of algebraic geometry tweaked and twisted to work in an arbitrary theory, for definable sets that have a Morley rank and functions between them. Thus \emph{irreducible} will mean Morley degree $1$, a notion that is insensitive to subsets of lower Morley rank. As there is no way to tell one definable bijection (e.g. linear) from another (e.g. Frobenius), we cannot avoid $p$th roots, and our \emph{degree} of a function is merely the number of points in a generic fiber, that is the separable degree. The beginning of section \ref{moshethmsection} is devoted to developing this language. Most of section \ref{forwref} is a translation from that language back to the usual algebraic geometry.

 Sections \ref{moshethmsection} and \ref{groupsection} are of general interest in that they describe the structure of models of $T_A$. Section \ref{moshethmsection} develops the basic tools of ACFA for $T_A$ and culminates with Theorem \ref{moshethm} characterizing $\lsig$-interalgebraicity in terms of $L$-definable sets and functions. Section \ref{groupsection} is a sequence of exercises is $\omega$-stable groups and simple groups that yields Theorem \ref{classifygp} that characterizes groups definable in models of $T_A$.

 Here is an outline of how we obtain our characterization of grouplike minimal types.
 Section \ref{moshethmsection} turns non-orthogonality between $L_\sigma$-types into a commutative diagram of $L$-definable sets and functions. Section \ref{groupsection} uses the main result of section \ref{moshethmsection} to obtain a commutative diagram with a group correspondence in it. Section \ref{chase} improves the diagram obtained in section \ref{groupsection} to the point where, in the special case of ACFA, it can be attacked with algebraic geometry. Then section \ref{forwref} describes this algebra-geometric attack, postponing the proofs of the crucial algebra-geometric results to a later paper \cite{secondpaper}.

  We use fairly naive and certainly well-known techniques, or close variants thereof: many, if not most, of our lemmas are there only to introduce the notation and tell the story. Nevertheless, some of the results (Theorem \ref{moshethm}, Lemma \ref{groupwrap}, Theorem \ref{classifygp}, and the final Theorem \ref{endthm}) are not entirely obvious and somewhat cute. The last theorem in particular proved crucial in \cite{polydyn} for classifying invariant subvarieties for certain algebraic dynamical systems.

\section{Many definitions and a technical theorem.}
\label{moshethmsection}

Most of this section is definitions and basic facts that are well known for ACFA. Here, we generalize them for $T_A$ (when it exists) for an arbitrary totally transcendental 
theory $T$ with quantifier elimination and elimination of imaginaries. We suspect that it should be possible to replace "totally transcendental" by "stable". We end this section with a technical, not terribly surprising, but quite useful Theorem \ref{geninteralg}.

\subsection{Finite dominant rational functions in $T$.}
\label{prelude}

 The results from this section will be used for $L$-definable sets; in particular, \textbf{``irreducible'' and ``rational'' imply $L$-definable}.

 Before doing anything else, we explain our notation for germs of definable functions between types. 
 If a formula defines a function on the set of realizations of a type $p$, then by compactness it also defines a function on some definable set $\DD \supset p$. By shrinking $\DD$, we may assume that it has the same Morley rank as $p$, and Morley degree $1$. Then for any other definable $\BB \supset p$ with $\RM(\BB) = \RM(p)$ and $\dM(B) = 1$ we have $\RM(\BB \cap \DD) = \RM(\BB)$. Instead of fixing a type $p$ and considering definable functions on its realization, we prefer to fix such a $\BB$ and consider definable functions whose domains are ``\zdense'' in $\BB$. This is a purely cosmetic and ideological difference. Our definitions are inspired by, and lifted from, algebraic geometry, but they are subtly different. We fix a language $L$ and an $L$-theory $T$.

\begin{deff}
\label{findom}
 An $L$-definable set $\AAA$ is \emph{irreducible} if 
 $\dM(\AAA) = 1$.

 If $B$ contains the parameters defining an irreducible $\AAA$,
then a subset $S \subset \AAA$ is \emph{\zdense in $\AAA$ over $B$} if $S \not\subset \CC$ for any $\CC \subset \AAA$ definable over $B$ with $\RM(\CC) < \RM(\AAA)$.

 For irreducible $\AAA$ and $\BB$, a \emph{rational function from $\BB$ to $\AAA$} is an $L$-definable function $f: \BB_0 \rightarrow \AAA$ whose domain $\BB_0$ is a \zdense subset of $\BB$.

A rational function $f$ from $\BB$ to $\AAA$ is \emph{finite} if there is an integer $n$ such that
$$\RM( \{ b \in \BB \sthat n =|f^{-1}(f(b))| \} ) = \RM(\BB)$$
 This integer $n$ is then called the \emph{degree} of $f$.

A rational function $f$ from $\BB$ to $\AAA$ is \emph{dominant} if the image of $f$ is \zdense in $\AAA$.

Two rational functions $f$ and $g$ from $\BB$ to $\AAA$ are \emph{equivalent}, written $f \approx g$, if the set $\{ x \sthat f(x) = g(x) \}$ on which $f$ and $g$ agree is \zdense in $\BB$, over the parameters needed to define $\AAA$, $\BB$, $f$ and $g$.

\textbf{``Irreducible'' and ``rational'' each imply $L$-definable, and ``finite'' and ``dominant'' each imply rational.}

\end{deff}

When we abuse notation and say ``$S \subset \AAA$ is \zdense in $\AAA$ over $B$'' even though $\AAA$ is not irreducible, we mean that $S$ is \zdense in any irreducible subset of $\AAA$ of full Morley rank.
 We now list a few trivial observations about these notions, in no particular order.
 If $S$ is \zdense in $\AAA$ over $B$, then it is also \zdense in $\AAA$ over any $B'\subset B$.

\begin{llama}
\label{ineverybig}
 A subset $S$ of an irreducible $\AAA$ is \zdense in $\AAA$ over $B$ if and only if
$S \cap \CC$ is non-empty for any $B$-definable $\CC \subset \AAA$ with $\RM(\CC)=\RM(\AAA)$.
 If in addition $S$ is a subset of a complete type, $S \subset \CC$ for any such $\CC$.
 If $S$ is definable (over some parameter set $C$), then it is \zdense in $\AAA$ over $C$ if and only if $\RM(S) = \RM(\AAA)$.
\end{llama}

If $S$ is \zdense in $\AAA$ and $T$ is \zdense in $\CC$, then $S \times T$ is \zdense in $\AAA \times \CC$.

\begin{llama} \label{defzdense}
 Suppose that $E$ is an $\omega$-saturated model of $T_A$. If $S$ is $\lsig$-partial-type-definable over $E$, and $\AAA$ is irreducible and definable over $E$, and $S(E)$ is \zdense in $\AAA$ over $E$, then $S(N)$ is \zdense in $\AAA$ over $N$ for any $N \succ E$.
\end{llama}

\begin{proof}
 Towards contradiction, suppose that $S \subset \CC$ for some $\CC$ defined by $\phi(x, c)$ for some $L$-formula $\phi(x,y)$ and some $c \in N$, with $\RM(\CC) < \RM(\AAA)$. Let $q(y)$ be the type of $c$ over $E$ and $q_1$ be the type of $c$ over the empty set. Now $(x \in S) \wedge q(y)$ implies $\phi(x,y)$, so by compactness $(x \in S) \wedge q_0(y)$ implies $\phi(x,y)$ for some finite $q_0 \subset q$. By saturation, there is some $c' \in E$ realizing $q_0 \cup q_1$, and now $\CC'$ defined by $\phi(x, c')$ contradicts the assumption that $S(E)$ is \zdense in $\AAA$ over $E$.
\end{proof}

 Because of this lemma we often abuse notation by saying ``\zdense'' without specifying the parameter set.

 If irreducible $\AAA$ and $\BB$ have the same Morley rank, then a rational function from $\BB$ to $\AAA$ is finite if and only if it is dominant. Somewhat conversely, if there is a finite dominant function between irreducible $\BB$ and $\AAA$, then they must have the same Morley rank.


\begin{llama}
 If $f_1, g_1 : \AAA \rightarrow \BB$ and $f_2, g_2: \BB \rightarrow \CC$ are finite dominant rational functions between irreducible sets, then both $g_2 \circ g_1$ and $f_2 \circ f_1$ are also finite dominant rational functions, and $\deg(g_2 \circ g_1) = \deg(g_2) \cdot \deg(g_1)$ .
 If in addition $f_i \approx g_i$ for $i = 1,2$, then also $g_2 \circ g_1 \approx f_2 \circ f_1$.
\end{llama}

We now prove three lemmas about compositional components of finite dominant rational functions.
\label{funnylemmas}

\begin{deff}
 If $f \approx h \circ g$, we say that $g$ is an \emph{initial factor} of $f$, and $h$ is a \emph{terminal factor} of $f$. Each factor is \emph{non-trivial} if it is not equivalent to a bijection, and \emph{proper} if the other factor is not equivalent to a bijection. We call $f$ an \emph{extension} of $g$ for lack of better term.
 \end{deff}

\begin{llama}
\label{shareinit}
 Any two finite rational functions $f: \AAA \rightarrow \BB$ and $g: \AAA \rightarrow \CC$ have a \emph{maximal shared initial factor} $h : \AAA \rightarrow \DD$ such that any initial factor of both $f$ and $g$, is an initial factor of $h$.
\end{llama}
\begin{proof}
 Let $h(x) := (f(x), g(x) )$.
\end{proof}

 By its universality, $h$ is unique up to composing with bijections on the left.

\begin{llama}
\label{leastclosing}
 If $T$ eliminates imaginaries, and $f$ and $g$ are both initial factors of the same finite $H$, then there exists a \emph{least common extension} of $f$ and $g$, that is, a finite definable $h$ such that $f$ and $g$ are initial factors of $h$, and any extension of both $f$ and $g$ is also an extension of $h$. In particular, $h$ is an initial factor of $H$, but does not depend on the choice of $H$, and is unique up to composing with bijections on the left.
\end{llama}
\begin{proof}
 Consider the binary relation $R_0$ given by $(f(a) = f(a') \vee g(a) = g(a'))$ and start constructing its transitive closure: define inductively $R_{n+1}(a,c)$ to be given by $\exists b\, R_0(a,b) \wedge R_n(b, c)$ and note that $R_{n+1} \supset R_n$.
 For any $a$, consider finite sets $S_n(a)$ defined by
 $$S_n(a) := \{ b \sthat (a,b) \in R_n \} \subset \{ b \sthat H(a) = H(b) \}$$
 On the definable subset $\EE$ of the domain of $H$ where $H$ has uniformly finite fibers, the sizes of $S_n(a)$ are bounded by the degree of $H$. So for each $a$ there is some $n$ such that $S_{n+1}(a) = S_n(a)$, and then $S_m(a) = S_n(a)$ for all $m > n$. By compactness, some $N$ works for all $a$, and then $R_N$ is an equivalence relation.
 Applying elimination of imaginaries to it yields the desired function $h$ with domain $\EE$.
 \end{proof}

 Although the least common extension of $f$ and $g$ produced in the lemma does not depend on the common extension $H$ we start with, the \emph{existence} of the least common extension very much depends on the \emph{existence} of some finite common extension. For example, in algebraic geometry the vast majority of finite rational morphisms $f$ and $g$ do not admit a common finite extension.
  Even when $f$ and $g$ do admit a common extension, its degree may be much higher than the degrees of $f$ and $g$. But in one special case we can get around this.

\begin{llama}
\label{groupclosing}
 Suppose that $T$ eliminates imaginaries; that $\AAA$, $\BB$, and $\CC$ are irreducible ($L$-definable) groups; and that  $f: \AAA \rightarrow \BB$ and $g: \AAA \rightarrow \CC$ are finite ($L$-definable rational) group homomorphisms. Then there is a finite rational group homomorphism $h$ which is a common extension of $f$ and $g$ with $\deg(h) \leq \deg(f) \cdot \deg(g)$.
\end{llama}

\begin{proof}
 The product $N := \{ a \cdot b \sthat f(a)=0 \wedge g(b)=0 \}$ of the kernels of $f$ and $g$ is a normal subgroup of $\AAA$ of size at most $\deg(f) \cdot \deg(g)$. Using elimination of imaginaries again, let $h$ be the quotient by $N$
\end{proof}

 We now say a few words about finite-to-finite correspondences, the central object of this paper.

\begin{deff}
 If in
  $$ \AAA \stackrel{f}{\leftarrow} \BB \stackrel{g}{\rightarrow} \CC$$
 the three definable sets are irreducible and the two rational functions are dominant,
 we say that the image $(f \boxtimes g) (\BB) \subset \AAA \times \CC$
 is a \emph{correspondence between $\AAA$ and $\CC$}.

 If in addition both $f$ and $g$ are finite, we say that the image is a
  \emph{finite-to-finite correspondence between $\AAA$ and $\CC$}.
\end{deff}

 We often abuse notation and say that $\BB$ is a correspondence from $\AAA$ to $\CC$; this is harmless.
 Finite-to-finite correspondences preserve many properties that definable bijections preserve. For example, two sets that admit a finite-to-finite correspondence between them must have the same Morley rank. It is often convenient to trim off the low Morley rank part of the correspondence where the projections do not have finite fibers.

\begin{llama}
\label{genericB0}
 If $$ \AAA \stackrel{f}{\leftarrow} \BB \stackrel{g}{\rightarrow} \CC$$
is a finite-to-finite correspondence, there is a \zdense subset $\BB_0 \subset \BB$ on which both $f$ and $g$ are defined and have finite fibers such that
 $$ \AAA \stackrel{f}{\leftarrow} \BB_0 \stackrel{g}{\rightarrow} \CC$$
is still a finite-to-finite correspondence.
\end{llama}

\begin{deff}
\label{groupcorrdef}
 When $\AAA$, $\BB$, and $\CC$ are definable groups, and $f$ and $g$ are definable group homomorphisms, the correspondence
 $$ \AAA \stackrel{f}{\leftarrow} \BB \stackrel{g}{\rightarrow} \CC$$
is called a \emph{group correspondence}.
\end{deff}

\subsection{Prolongations, Sharps, and \emph{very dense subsets}.}

 Now we are back to the setting with two languages $\lsig = L \cup \{ \sigma \}$ and two theories $T$ and $T_A$.

\begin{deff}

 For $S \subset \MM \models T_A$, we write $S^\sigma$ for $\{ \sigma(s) \sthat s \in S \}$.

 For $S \subset \MM \models T_A$, the set $S^+ := \{ (s, \sigma(s)) \sthat s \in S \}$ is called the \emph{first prolongation} of $S$.

 More generally, the \emph{$n$th prolongation} of $S$ is $\{ (s, \sigma(s), \ldots \sigma^n(s)) \sthat s \in S \}$.
\end{deff}

 The second prolongation is not first prolongation of the first prolongation.
If $S$ is definable by an $\lsig$-formula with parameters $a$, then $S^\sigma$ is defined by the same formula with parameters $\sigma(a)$, as $\sigma$ is an $\lsig$-automorphism of a model of $T_A$.
 In particular, when $S$ is $L$-definable, then so is $S^\sigma$.
 Since projection onto the first coordinate is a definable bijection from $S^+$ to $S$, all $L_{\sigma}$-properties invariant under definable bijection, such as ranks, pass to prolongations. The natural bijection $x \mapsto (x, \sigma(x))$ from $S$ to $S^+$ is also called a prolongation.

\begin{deff}
 If an $L$-definable set $\BB$ comes with rational, $L$-definable functions $f$ and $g$ to $\AAA$ and $\AAA^\sigma$, (for example, when $\BB \subset \AAA \times \AAA^\sigma$ and $f$ and $g$ are projections), we write
 $(\AAA,\BB)^{sh} := \{ a \sthat (a, \sigma(a)) \in (f \boxtimes g) \BB\}$.

 When $\BB$ is a finite-to-finite correspondence between $\AAA$ and $\AAA^\sigma$, we write $(\AAA,\BB)^\sharp$ instead of $(\AAA,\BB)^{sh}$.

  \textbf{Whenever we write $(\AAA,\BB)^\sharp$, we are making this assumption, namely that $\AAA$ and $\BB$ have Morley degree $1$, and the two functions from $\BB$ to $\AAA$ and $\AAA^\sigma$ are finite, dominant rational functions.}

 When $\BB$ is the graph of some function $f: \AAA \rightarrow \AAA^\sigma$, we write $(\AAA, f)^\sharp$ instead of $(\AAA, \BB)^\sharp$.

\end{deff}

\begin{llama}
\label{sharpacl}
 For any $(\AAA,\BB)^\sharp$ there is a \zdense $\BB_0 \subset \BB$ such that  $\aclsig(s) = \acl(s)$ for any $s \subset (\AAA,\BB_0)^\sharp$.
\end{llama}

\begin{proof}
 We obtain $B_0$ from Lemma \ref{genericB0}.
 For any $a \in (\AAA,\BB_0)^\sharp$, $\sigma(a)$ and $\sigma^{-1}(a)$ are already $L$-algebraic over $a$, and the conclusion follows from Fact \ref{basicfacts}(2).
\end{proof}

\begin{deff}
 $S \subset (\AAA,\BB)^\sharp$ is \emph{very dense over $E$} if $S$ is \zdense in $\AAA$ over $E$, or, equivalently, if the first prolongation $S^+$ is \zdense in $\BB$ over $E$.
\end{deff}

\begin{prop}
\label{sharpdense}
 In any model $E \models T_A$, the definable set $(\AAA,\BB)^\sharp$ is very dense in itself over the whole model $E$.
\end{prop}

\begin{proof}

 We translate the proof of Theorem 1.1 in \cite{ChaHru1} into our notation and note that it has nothing to do with fields.

 Suppose $\AAA$, $\BB$, and $\AAA^\sigma$ are $L$-definable over $E \models T_A$.

 First note that for any definable $S$, the statement ``$S$ is \zdense in $\AAA$'' is a conjunction of a set of existential first-order formulae over $E$: for each $L$-definable low-Morley-rank $\CC$, we demand that there exist something in $S$ but not in $\CC$. (Remember that $T_A$ is model-complete, so $S$ is defined by an existential formula; and $T$ eliminates quantifiers, so $\CC$ is defined by a quantifier-free formula.)

 Since models of $T_A$ are existentially closed in the class of models of $T_\sigma$, it is sufficient to find a model $N$ of $T_\sigma$ such that $E$ embeds into $N$ and  $(\AAA,\BB)^\sharp( N)$ is very dense in $(\AAA,\BB)^\sharp$ over $E$.

 We start with a sufficiently saturated model $M$ of $T$ such that $E \upharpoonleft_L$ embeds into $M$. We find in $M$ an element $a$ realizing the (unique!) generic $L$-type $t_\AAA$ of $\AAA$. We find $a'$ such that $(a, a') \in \BB$. Since $\BB \subset \AAA \times \AAA^\sigma$, $a'$ belongs to the $L$-definable set $\AAA^\sigma$. The generic $L$-type of $\AAA^\sigma$ is ${t_\AAA}^\sigma$. Since $a$ and $a'$ are $L$-interalgebraic over $E$, they have the same Morley rank, so $a'$ realizes the generic $L$-type of $\AAA^\sigma$. In other words, the map $\tau$ that takes $E$ to itself by $\sigma$, and takes $a$ to $a'$ is a partial-$L$-elementary map from $M$ to $M$. By the saturation of $M$, we can extend $\tau$ to make it an $L$-automorphism of an $L$-substructure $N$ of $M$.

Now $(N, \tau) \models T_\sigma$ is the model we want, and $(a,a') \in (\AAA,\BB)^\sharp (N)$ is all by itself very dense in $(\AAA,\BB)^\sharp$ over $E$.
 \end{proof}

When $M \prec N \models T_A$ and $S$ is $\lsig$-definable over $M$ and $(\AAA, \BB)^\sharp$ is defined over $M$, a similar argument show that $S(M)$ is very dense in $(\AAA,\BB)^\sharp$ over $M$ if and only if $S(N)$ is very dense in $(\AAA,\BB)^\sharp$ over $N$. If $S$ is only type-definable, the same holds for a slightly-saturated $M$. Thus, we often abuse notation by saying ``very dense'' without specifying the parameter set.

\subsection{$\sigma$-degree.}

 The notion of $\sigma$-degree has been used in difference algebra for many decades, and translates naturally to our setting. However, the tight connection between $\sigma$-degree and $U$-rank in ACFA does not generalize to $T_A$.


\begin{deff}
 Let $M \models T_A$, $A \subset M$, $a \in M$;
 if for some $N$,
 $$\sigma^{N}(a) \in \acl( A \cup \{a, \sigma(a), \ldots \sigma^{N-1}(a) \})$$
then the \emph{$\sigma$-degree of $a$ over $A$} is
 $$\deg_\sigma(a/A) := \RM((a, \sigma(a), \ldots \sigma^N(a)) /\aclsig(A) )$$
that is,
 $$ \min \{ \RM(\BB) \sthat (a, \sigma(a), \ldots \sigma^N(a)) \in \BB \mbox{ and $\BB$ is $L$-definable over $\acl_\sigma(A)$} \}$$
 and we say that $\deg_\sigma(a/A)$ is \emph{defined}.
  Otherwise, we say that $\deg_\sigma(a/A) = \infty$.
\end{deff}

It is clear that $\sigdeg(a/A)$ is the property of the quantifier-free $L_\sigma$-type of $a$ over $\aclsig(A)$; as usual, we can extend the definition to partial types, including definable sets, by taking the supremum over realizations in a sufficiently large model. As with Morley rank, this supremum is attained in a sufficiently saturated model.
 It is also clear that for any $B \supset A$ we have $\sigdeg(a/B) \leq \sigdeg(a/A)$, since on the left hand side the minimum is taken over a larger family of definable sets, and $N$ also only decreases.

\begin{llama} \label{sigdegprolong}
 The $\sigma$-degree is invariant under prolongation.
\end{llama}
\begin{proof}
 This comes down to obtaining the $m$th prolongation of the $n$th prolongation from the $(m+n)$th prolongation by adding new variables and carefully setting them equal to appropriate old ones.
\end{proof}

\begin{llama}\label{prolongvdense}
 Suppose that $S$ is defined by a complete $\lsig$-type over a model $E$. If $S$ has defined $\sigma$-degree over $E$, then some prolongation of $S$ is very dense in some $(\AAA, \BB)^\sharp$ defined over $E$.
\end{llama}

\begin{proof}
 Take the $(N-1)$st prolongation of $S$ for $N$ from the definition of $\sigma$-degree, find some (irreducible for $S$ a complete type over a model) $\AAA$ in which the $(N-1)$st prolongation of $S$ is \zdense, and get $\BB$ from the formula witnessing that $\sigma^{N}(a) \in \acl( E \cup \{a, \sigma(a), \ldots \sigma^{N-1}(a) \})$.
\end{proof}

\begin{llama}
 When $\sigma$-degree is defined, it witnesses forking.
\end{llama}
\begin{proof}
 We need to show that if $\sigdeg(a/C)$ is defined, then $a$ is independent from $B$ over $C$ if and only if $\sigdeg(a/BC) = \sigdeg(a/C)$.

 We may assume without loss of generality that $C = \acl_\sigma(C)$ and $B = \acl_\sigma(B)$ and $C \subset B$. Since $\sigdeg(a/C)$ is defined, let $N$ be such that $\sigma^{N}(a) \in \acl( C \cup \{a, \sigma(a), \ldots \sigma^{N-1}(a) \})$, and note that $\aclsig(Ca) = \acl( C \cup \{a, \sigma(a), \ldots \sigma^{N-1}(a) \})$.

  From Fact \ref{basicfacts}.4, $a$ is independent from $B$ over $C$ in the sense of $T_A$ if and only if
 $\acl_\sigma(Ca)$ is independent from $B$ over $C$ in the sense of $T$, which happens if and only if $\RM( (a, \sigma(a), \ldots \sigma^{N-1}(a)) /B) = \RM( (a, \sigma(a), \ldots \sigma^{N-1}(a)) /C)$, which happens if and only if $\sigdeg(a/B) = \sigdeg(a/C)$, and we are done.
\end{proof}

If follows immediately that
\begin{sheep}
 The $\sigma$-degree is an upper bound on $U$-rank.
\end{sheep}

 In ACFA, there is a partial converse to this: finite $U$-rank implies finite $\sigma$-degree (2.5 in \cite{ChaHru1}). However, that proof relies on ACF having finite rank and on the whole model of ACF being a group. We expect that only the second of these is necessary.

\begin{llama}
\label{sloganlemma}
 Slogan: The $\sigma$-degree of any $\lsig$-partial-type-definable very dense subset of $(\AAA,\BB)^\sharp$ is $\RM(\AAA)$.

 Statement: Suppose that $\AAA$ and $\BB$ are $L$-definable over $D$, and $S \subset (\AAA,\BB)^\sharp$ is $\lsig$-partial-type-definable over $D$, and very dense in $(\AAA,\BB)^\sharp$ over $\aclsig{D}$.
 Then $\sigdeg(S/D) = \RM(\AAA)$; if $D$ is an $\omega$-saturated model of $T_A$, then also $\sigdeg(S/C) = \RM(\AAA)$ for any $C \supset D$.
\end{llama}

\begin{proof}
 Since $\BB$ is a finite-to-finite correspondence between $\AAA$ and $\AAA^\sigma$, Lemma \ref{sharpacl} applies, and any $a \in (\AAA,\BB_0)^\sharp$ has finite $\sigma$-degree over any set containing $D$, witnessed by $N=1$.
 Also, $\sigdeg(a/C) \leq \sigdeg(a/ D) \leq RM(\AAA)$ as $\AAA$ is $D$-definable and contains $a$. So it suffices to show that $\sigdeg(S/C) \geq RM(\AAA)$. But $S$ is \zdense in $\AAA$ over $D$ (and if $D$ is an $\omega$-saturated model, then also over any $C \supset D$ by Lemma \ref{defzdense}), so $S \not\subset \AAA'$ for any $C$-definable $\AAA'$ with $\RM(\AAA') \lneq \RM(\AAA)$.
\end{proof}

Lemma \ref{prolongvdense} makes some prolongation of an $\lsig$-definable set with defined $\sigma$-degree \zdense in some correspondence, and Lemma \ref{sigdegprolong} asserts that its $\sigma$-degree is unaffected by the prolongation, so the last lemma is almost a defining property of $\sigma$-degree. This is used (and explained in more detail) in Corollary \ref{a3} below.


 We will be interested in complete $\lsig$-types whose realizations are very dense in $(\AAA, \BB)^\sharp$.
If $\AAA$ is a curve (that is, Morley rank $1$), any non-algebraic type is very dense. More generally, the realizations of a type $p$ are very dense in $(\AAA, \BB)^\sharp$ whenever $\RM(p|_L) = \RM(\AAA)$. The relationship between the properties ``$p$ is very dense in $(\AAA, \BB)^\sharp$'' and $U(p) = U((\AAA, \BB)^\sharp)$ is most intriguing.

\subsection{The theorem.}

\begin{deff}
 We say that two sets $S$ and $T$ are \emph{uniformly interalgebraic} if there is a formula $\theta(x;y)$ and an integer $N$ such that
$$\forall s \in S \; \exists_{\leq N}^{\geq 1} t \in T \; \theta(s,t)
 \mbox{ and } \forall t \in T \; \exists_{\leq N}^{\geq 1} s \in S \; \theta(s,t)$$
\end{deff}

There are hidden assumptions in the following theorem: don't forget the bold text at the end of Definition \ref{findom}.

\begin{thm}
\label{geninteralg}
\label{moshethm}
 Suppose that $\AAA$, $\BB$, $\CC$, and $\DD$ are $L$-definable sets, 
 that $S$ and $T$ are defined by complete $\lsig$-types,
 and that $\theta$ is an $\lsig$-formula such that
 \begin{itemize}
\item $S$ is a very dense subset of $(\AAA,\BB)^\sharp$;
\item $T$ is a very dense subset of $(\CC,\DD)^\sharp$;
\item $\theta(x;y)$ witnesses that $S$ and $T$ are uniformly interalgebraic.
\end{itemize}

 Then then there is a (quantifier-free) formula $\zeta(x, x';y, y') \in L$
such that
 \begin{enumerate}
 \item $\theta(x,y) \wedge x \in S \wedge y \in T$ implies
  $\zeta(x, \sigma(x);y, \sigma(y))$;
 \item $\zeta(x, x', y, y')$ implies $(x, x') \in \BB$ and $(y,y') \in \DD$, and
  $$ \{ (a, a', c, c') \sthat \zeta(a,a',c,c') \} =: \FF \subset \BB \times \DD$$ is a finite-to-finite correspondence between $\BB$ and $\DD$.
 \item $\EE := \{ (a,c) \sthat \exists a', c'\, (a,a', c, c') \in \FF \}$ is a finite to finite correspondence between $\AAA$ and $\CC$, and
 $\{ (a',c') \sthat \exists a, c\, (a,a', c, c') \in \FF \} = \EE^\sigma$
 \end{enumerate}
 \end{thm}

\begin{proof}
 We may assume without loss of generality that $(\AAA,\BB)^\sharp$ and $(\CC,\DD)^\sharp$ already satisfy the conclusion of Lemma \ref{genericB0}.

 By compactness, we can find $\lsig$-definable $S_1$ and $T_1$ such that
$S \subset S_1 \subset (\AAA,\BB)^\sharp$ and $T \subset T_1 \subset (\CC,\DD)^\sharp$, with $S_1$ and $T_1$ still uniformly interalgebraic via $\theta$; note that $S_1$ and $T_1$ are still very dense in $(\AAA,\BB)^\sharp$ and $(\CC,\DD)^\sharp$, respectively.

 By Lemma \ref{sharpacl} and Fact \ref{basicfacts}(2) there are $L$-formulae $\phi(x, x',y, y')$ such that \begin{itemize}
\item $x \in S_1 \wedge y \in T_1 \wedge \theta(x, y)$ implies
 $\phi(x, \sigma(x),y, \sigma(y))$, and
\item there is a bound $M \in \mathbb{N}$ such that\\
 $\forall x \in S_1 \exists_{\leq M} (y,y') \, \phi(x, \sigma(x),y, y')$ and
  $\forall y \in T_1 \exists_{\leq M} (x,x') \, \phi(x, x', y, \sigma{y})$
\end{itemize}
 We take $\zeta_0(x,x',y,y')$ to be one of these $\phi$ with the least possible Morley rank.

Now $\BB_1 := \{ (a,a') \in \BB \sthat \exists_{<N} (y,y') \, \zeta_0(a, a',y, y') \}$ is an $L$-definable subset of $B$ containing the first prolongation of $S_1$, which is \zdense in $\BB$, so $\RM(\BB_1) = \RM(\BB)$. Define $\DD_1$ the same way, and make the same observation.

Let $\zeta(x,x',y, y') := \zeta_0(x,x',y,y') \wedge (x,x') \in \BB_1 \wedge (y,y')\in \DD_1$, and let $\FF \subset \BB \times \DD$ be defined by it. Note that the image of the projection $\FF \rightarrow \BB$ contains $S_1$ and therefore has full Morley rank, and similarly for the image in $\DD$.

We have now shown that $\FF$ is a finite-to-finite correspondence between $\BB$ and $\DD$, so (2) is proved.

As for the first conclusion, we already have that
$$x \in S_1 \wedge y \in T_1 \wedge \theta(x, y) \mbox{ implies }
 \zeta_0 (x, \sigma(x),y, \sigma(y))$$
Since $S_1 \supset S$ and $T_1 \supset T$, it follows that
$$x \in S \wedge y \in T \wedge \theta(x, y) \mbox{ implies }
 \zeta_0 (x, \sigma(x),y, \sigma(y))$$
Since $\BB_1 \subset \BB$ has full Morley rank, and the first prolongation of $S$ is a complete type \zdense in $\BB$, Lemma \ref{ineverybig} shows that $x \in S$ implies $(x,\sigma(x)) \in \BB_1$. The identical argument for $T$ in $\DD$ finishes the proof of (1).

To see that $\FF$ projects dominantly onto $\AAA$ and $\CC$, note that a composition of finite dominant rational functions is itself finite dominant. So $\EE$ is indeed a finite to finite correspondence between $\AAA$ and $\CC$.
 Since $(S^+ \times T^+) \cap \FF$ is \zdense in $\FF$, its projections $(S \times T)$ and $(S^\sigma \times T^\sigma)$ are \zdense in the two (finite!) projections of $\FF$, finishing the proof of (3).
\end{proof}

 It is worth noting that the conclusion of this theorem cannot be sharpened to make $\theta$ and $\zeta$ equivalent on $S \times T$: for example, $\theta$ may be the graph of $\sigma$.

\begin{sheep}
\label{moshediag}
 If two $\lsig$-types $p$ and $q$, both of U-rank $1$, are non-orthogonal,
  and $p$ is a very dense subset of $(\AAA,\BB)^\sharp$, and
$q$ is a very dense subset of $(\CC,\DD)^\sharp$,
 then $\RM(\AAA) = \RM(\CC)$ and there  are $L$-definable $\EE$ and $\FF \subset \EE \times \EE^\sigma$, and a U-rank $1$ type $r \in (\EE,\FF)^\sharp$ and finite dominant rational $\pi: \EE \rightarrow \AAA$, and $\rho: \EE \rightarrow \CC$ such that $\pi(r) = p$ and $\rho(r) = q$ and the following diagram commutes

 \begin{equation*}
\begin{CD}
\CC @<<< \DD @>>> \CC^{\sigma}\\
@AA{\rho}A @AAA @AA{\rho^\sigma}A\\
\EE @<<< \FF @>>> \EE^{\sigma}\\
@VV{\pi}V @VVV @VV{\pi^\sigma}V\\
\AAA @<<<  \BB @>>> \AAA^\sigma\\
\end{CD}
\end{equation*}

\end{sheep}
\begin{proof}
We suppress parameters - either we must begin with a sufficiently saturated model as our parameter set, or we must allow the possibility that the new sets need new parameters. This permits us to equate orthogonality and almost-orthogonality.


\end{proof}

Note that this does not make $(\AAA,\BB)^\sharp$ definably interalgebraic with $(\CC,\DD)^\sharp$: it may easily be that $\pi( (\EE,\FF)^\sharp)$ is a proper subset of $(\AAA,\BB)^\sharp$, witnessing the lack of full quantifier elimination in $T_A$.

\begin{sheep} \label{a3}
Uniformly interalgebraic $\lsig$-type-definable sets with defined $\sigma$-degree have the same $\sigma$-degree.
Non-orthogonal $U$-rank $1$ $\lsig$-types with defined $\sigma$-degree have the same $\sigma$-degree.
\end{sheep}

\begin{proof}
For the first part, let $S_0$ and $T_0$ be the uniformly interalgebraic $\lsig$-definable sets with defined $\sigma$-degree. Use Lemma \ref{prolongvdense} to find prolongations $S$ of $S_0$ very dense in $(\AAA, \BB)^\sharp$, and $T$ of $T_0$ very dense in $(\CC, \DD)^\sharp$. By Lemma \ref{sloganlemma}, $\sigdeg(S) = \RM(\AAA)$ and $\sigdeg(T) = \RM(\CC)$.
Clearly, $S$ and $T$ are still uniformly interalgebraic (witnessed by the same formula).
By Theorem \ref{moshethm}, there is a finite-to-finite correspondence between $A$ and $C$, so $\RM(\AAA) = \RM(\CC)$. By Lemma \ref{sigdegprolong}, $\sigdeg(S_0) = \sigdeg(S)$ and $\sigdeg(T_0) = \sigdeg(T)$.

For the second part, note as in the last corollary that non-orthogonal types of $U$-rank $1$ are uniformly interalgebraic.
\end{proof}

 One special case of uniformly interalgebraicity is an $\lsig$-definable bijection.
 Clearly, a type with defined $\sigma$-degree cannot be interalgebraic with a type with undefined $\sigma$-degree.

\section{Groups in $T_A$.}
\label{groupsection}
 We begin this section with Corollary \ref{gpdiagthm}, a statement about ACFA from the author's thesis (Theorem 3 in \cite{mythesis}) that motivated this work. We then explain the correct statement in the general setting of $T_A$ and prove it.

\subsection{Motivation from ACFA}

Remember, $U$ is the Lascar rank of $\lsig$-definable or $\lsig$-type-definable sets, and we write $U(\pi)$ for a partial type $\pi$ over $A$ to mean $\sup \{ U(p) \sthat \pi \subset p \in S(A) \}$.

\begin{deff}
 We call a $U$-rank 1 type or definable set \emph{minimal}, even though ``weakly minimal'' is more correct.

 Following the terminology in \cite{ChaHru1}, we say that a minimal type $p$ is \emph{modular} (over some $E$) if whenever $A$ and $B$ are sets of realizations of $p$, $A$ and $B$ are independent over $\aclsig(EA) \cap \aclsig(EB)$.
 
 A type $p$ is \emph{locally modular} if for some $E$, some non-forking extension of $p$ to $E$ is modular over $E$.

 As usual, a minimal set $S$ is called \emph{trivial} if $\aclsig(A) = \cup_{a \in A} \aclsig(a)$ for any $A \subset S$.

 We call a minimal, locally modular, non-trivial type \emph{grouplike}. If all types in a minimal definable set are grouplike, we call the set itself grouplike.
\end{deff}

\begin{fact} (Zilber Trichotomy for ACFA \cite{ChaHruPet})\\
 In ACFA, every minimal type is exactly one of the following: \begin{itemize}
 \item non-orthogonal to a generic type of the fixed field of a definable automorphism (and therefore not locally modular);
 \item grouplike and non-orthogonal to a generic type of a definable minimal modular group;
 \item trivial.
\end{itemize}
\end{fact}

 The following statement about grouplike minimal types in ACFA (Theorem 3 in \cite{mythesis}) was the original motivation for this work. 

\begin{sheep} (ACFA)
\label{gpdiagthm}
 If some very dense $\lsig$-type $p$ in $(\AAA, \BB)^\sharp$ is grouplike
, then there is a group correspondence $(\EE,\FF)^\sharp$, and further irreducible sets $\CC$ and $\DD$ and $L$-definable finite dominant functions such that the following diagram commutes:
 \begin{equation*}\begin{CD}
\EE @<<< \FF @>>> \EE^{\sigma}\\
@AA{\rho}A @AAA @AA{\rho^\sigma}A\\
\CC @<<< \DD @>>> \CC^{\sigma}\\
@VV{\pi}V @VVV @VV{\pi^\sigma}V\\
\AAA @<<< \BB @>>> \AAA^{\sigma}\\
\end{CD}\end{equation*}

 In this diagram, the horizontal arrows are projections, since $\BB \subset \AAA \times \AAA^\sigma$, $\DD \subset \CC \times \CC^\sigma$, and $\FF \subset \EE \times \EE^\sigma$, and the two middle vertical arrows are restrictions of $\pi \times \pi^\sigma$ and $\rho \times \rho^\sigma$ to $\DD$.

\end{sheep}

 Corollary \ref{moshediag} provides the diagram once we find a minimal type $q$ nonorthogonal to $p$ and very dense in $(\EE,\FF)^\sharp$.

 Since we do not get full interalgebraicity in Theorem \ref{moshethm}, we do not get all types in $(\AAA,\BB)^\sharp$ to be grouplike; but we do get a large, quantifierfully definable subset $\pi( (\CC,\DD)^\sharp)$ of $(\AAA,\BB)^\sharp$ to be interalgebraic with a chunk of a group. We hope to prove one day that all very dense types in $(\AAA,\BB)^\sharp$ must be grouplike if one is, at least if the difference ideal generated by $(\AAA,\BB)^\sharp$ is prime. For the case when $\BB$ is the graph of a function, this is accomplished in the author's thesis \cite{mythesis} and described in this paper.

 It turns out that minimality, modularity, and the whole Zilber Trichotomy are irrelevant: Corollary \ref{gpdiagthm} is an easy consequence of Corollary \ref{gpdiagta} below. The correct hypothesis in the general setting is not that some very dense type in $(\AAA, \BB)^\sharp$ is grouplike, i.e. minimal and non-orthogonal to a generic type of an $\lsig$-definable minimal modular group, but only that this type is interalgebraic with a generic type of an $\lsig$-definable group. Minimality is replaced by defined $\sigma$-degree, automatic for a type in $(\AAA, \BB)^\sharp$. While the relationship between $U$-rank and $\sigma$-degree is still unclear in the general setting, the natural test case of DCFA confirms that defined $\sigma$-degree is the correct hypothesis: when the original theory $T$ does not have finite rank, assuming finite $U$-rank is far too restrictive.

We now set out to prove Corollary \ref{gpdiagta}. Along the way, we obtain a characterization of $\lsig$-definable groups with defined $\sigma$-degree (Theorem \ref{classifygp}), an analog of the characterization of groups of finite $U$-rank definable in ACFA. We do not return to ACFA until the very end of this section, where we prove Corollary \ref{gpdiagthm}.

\subsection{The key technical proposition.}

This section is devoted to the proof of Proposition \ref{groupkey} via some elementary facts about $\omega$-stable and simple groups.


 The next three lemmas are all about an abstract subgroup $H$ of an irreducible $L$-definable group $\AAA$. Later, we will assume that $H$ is $\lsig$-definable, but for the next three lemmas we work in one language $L$, inside a totally-transcendental group.

\begin{llama} (Entirely in $L$)\\
 Suppose that \begin{itemize}
\item $\AAA$ is an $L$-definable group;
\item  $\BB$ is an irreducible $L$-definable subset of $\AAA$;
\item  $r$ is the global generic $L$-type of $\BB$;
\item $\CC$ is the stabilizer of $r$ in $\AAA$; and
\item $H$ is an abstract subgroup of $\AAA$ which is \zdense in $\BB$;
\end{itemize}
 Then $\RM(\CC) = \RM(\BB)$, the group $\CC$ is connected, and $H \subset \CC$.
\end{llama}

\begin{proof}
 It is a standard fact about totally transcendental groups (\cite{bigPillay} 1.6.16 and 1.6.21) that $\CC$ is a definable subgroup of $\AAA$. If $h \in H$, then $H = h H$ is \zdense in $h \BB$, so $h \BB$ intersects $\BB$ in a subset of full Morley rank, so $h r = r$. Therefore $H \subset \CC$, so $\BB \cap \CC$ has the same Morley rank as $\BB$, so $r$ is in $\CC$.
 Now $r$ is a type in a stable group $C$ with $\operatorname{Stab}_C (r) = C$, so $C$ is connected and $r$ is its unique generic (\cite{bigPillay} 1.6.6.ii and 1.6.16.iii).


 \end{proof}

 The next lemma eliminates one of the hypotheses from the previous lemma (the irreducibility of $\BB$), at the cost of passing from $H$ to a finite-index subgroup.

\begin{llama}
 Suppose that $\AAA$ is an $L$-definable group 
  and $H$ is an abstract subgroup of $\AAA$.
 Let $\BB$ be a definable (perhaps with new parameters) subset of $\AAA$ containing $H$, with the least possible Morley rank $\alpha$ and degree $r$.
 Then there is a finite index subgroup $H'$ of $H$ which is \zdense in an irreducible $L$-definable $\BB'$ with $\RM(\BB') = \alpha$.
\end{llama}

\begin{proof}
 For any $h \in H$, $(h \BB) \cap \BB \supset H$ and so has the same Morley rank and degree as $\BB$. That is, translating by $h$ permutes the $r$ generic types of $B$, giving a homomorphism from $H$ into a finite group $S_r$; let $H'$ be the kernel of that homomorphism, a finite-index subgroup of $H$. Let $\BB'$ be a least Morley rank and degree definable set containing $H'$. Now a finite union of translates of $\BB'$ covers $H$, so $\BB'$ has the same Morley rank as $\BB$. So all generic types of $\BB'$ are also generic in $\BB$, and therefore fixed by all elements of $H'$.
 Write $\BB' = \cup_i \CC_i$ for disjoint irreducible $\CC_i$ of full Morley rank, and let $p_i$ be the generic type of $\CC_i$, and let $\CC_0$ contain the identity of the group. Since $H'$ is \zdense in $\BB'$, there are $h_i \in \CC_i \cap H'$ for each $i$. Then on one hand, $h_i \cdot p_0 = p_0$ since $H'$ fixes all generic types of $\BB'$; but on the other hand, $h_i \cdot p_0 = p_i$ since it is inside $C_i$. So $\BB'$ is irreducible.
\end{proof}

The purpose of all that was

 \begin{llama}
\label{groupwrap}
 Suppose $\AAA$ is an $L$-definable group 
  and $H$ is an abstract subgroup of $\AAA$. Let $\alpha$ be the least Morley rank of
a definable set containing $H$. Then there exists a finite-index subgroup $H' \leq H$, and an irreducible $L$-definable \emph{subgroup} $\CC$ of $\AAA$ of Morley rank $\alpha$ containing $H'$.
\end{llama}

\begin{proof}
 The last lemma gives a finite-index subgroup $H'$ of $H$ which is \zdense in an irreducible definable subset $\BB'$ of $\AAA$ of Morley rank $\alpha$. The lemma before that then gives the irreducible subgroup $\CC$ of $\AAA$ with $H'$ \zdense in $\CC$.
\end{proof}

 Now we turn to the case where $H$ is $\lsig$-definable and, after three more intermediate results, prove Proposition \ref{groupkey}.

\begin{llama}
\label{genericzdense}
 Suppose that $M$ be a small model of $T_A$, $\EE$ is an $L(M)$-definable group, $H \leq \EE$ is an $\lsig(M)$-definable subgroup which is \zdense in $\EE$ over $M$, and $q$ is a global generic $\lsig$-type of $H$. Then $q$ is \zdense in $\EE$ over $M$.
 \end{llama}

 \begin{proof}
Suppose, towards contradiction, that there is some $L(M)$-definable $\CC$ with $q \subset \CC \subset \EE$ and $\RM(\CC) \lneq \RM(\EE)$. Since $q$ is a complete type, we may and do assume that $\CC$ is irreducible. For any $h \in H$, the type $hq$ contains the formula $(x \in h\CC)$ and is also generic in $H$, so $hq$ does not fork over $M$. 
We find one $h \in H$ such that $(x \in h\CC)$ forks over $M$, obtaining the desired contradiction.

The formula $(x \in \CC)$ is not in the global generic $L$-type $p_\EE$ of $\EE$, so some $\EE$-translate of it forks over $M$. To find an $H$-translate of $(x \in \CC)$ that forks over $M$, we construct a Morley sequence $<e_i>_{i \in \omega}$ (in the sense of $L$) in\\ $p_0 := (p_\EE \mbox{ restricted to } M)$, with $e_0 \in H$. It is sufficient to find some $e_0 \in H$ realizing $p_0$. Since $H$ is \zdense in $\EE$ over $M$, the formula defining $H$ is consistent with $p_0$. By the saturation of the monster model, we find a realization $e_0$. We now show that the $L$-formula $x \in e_0 \CC$ forks over $M$.

Let $p_\CC$ be the generic $L$-type of $\CC$, and let $\BB$ be its stabilizer in $\EE$. Since $p_\CC$ is not generic in $\EE$ (lower Morley Rank), its stabilizer is a proper subgroup of $\EE$ of infinite index. In particular, $e_i^{-1}e_j \notin \BB$, so $e_i p_\CC$ are all distinct, so $\RM( (e_i  \CC) \cap (e_j \CC) ) \lneq \RM(e_i \CC)$. The following exercise completes the proof.

  \emph{Claim}: Suppose that $M$ is a small model, $\phi(x;y)$ is
$M$-definable (in this lemma, $\phi(x;y) := x \in y \CC$), and $\{
e_i \}_{i \in \omega} $ is an $M$-indiscernible sequence such that
$\phi(x, e_i) \wedge \phi(x, e_j)$ has strictly lower Morley rank
than $\phi(x, e_i)$. Then $\phi(x, e_0)$ forks over $M$.

 \emph{Proof of Claim}: Suppose towards contradiction that $\phi(x, e_0)$ does not
fork over $M$. Let $N$ be a bigger model, containing $A$ and all
the $e_i$. Let $p \in S(M)$ be the unique type in over $M$
containing $\phi(x, e_0)$ with $RM(p) = RM(\phi(x, e_0))$.  Let $q
\in S(N)$ be the unique nonforking extension of $p$, that is to
say the unique type over $N$ containing $\phi(x, e_0)$ with $RM(q)
= RM(\phi(x, e_0))$. Since $q$ does not fork over $M$, it is
definable over $M$; let $\theta(y) \in L(M)$ be its definition
with respect to $\phi(x;y)$. Since $\phi(x, e_0) \in q$, it
follows that $theta(e_0)$ holds. Since $\{ e_i \}_{i \in \omega} $
is $M$-indiscernible and $theta$ is over $M$, this implies that
$\theta(e_i)$ holds for all $e_i$. That means that $\phi(x, e_i)
\in q$ for all $i$, and then $\phi(x, e_0) \wedge \phi(x, e_1)$ is
also in $q$. But then $RM ( \phi(x, e_0) \wedge \phi(x, e_1) )
\lneq RM(\phi(x, e_0)) = RM(q)$ gives the desired contradiction.
\end{proof}

 \begin{sheep}
\label{fininddense}
 Suppose that $H$ is an $\lsig$-definable subgroup of a group correspondence $(\EE, \FF)^\sharp$ and $H$ is very dense in $(\EE, \FF)^\sharp$. Then any generic $\lsig$-type $q$ of $H$ and any finite-index $\lsig$-definable subgroup $K$ of $H$ are also very dense in $(\EE, \FF)^\sharp$.

\end{sheep}
\begin{proof}
 If $H$ is very dense in $(\EE, \FF)^\sharp$, then it is \zdense in $\EE$, so the statement about $q$ is an immediate corollary of the last lemma. The rest follows because every finite-index subgroup contains a generic type.
\end{proof}

\begin{llama}
\label{verydensegeneric}
 If an $\lsig$-definable subgroup $H_2$ of a group correspondence $(\EE, \FF)^\sharp$ is very dense, then it has finite index.
\end{llama}

\begin{proof}
 Suppose towards contradiction that $H_2$ has infinite index in $(\EE, \FF)^\sharp$. Let $M$ be a small model over which everything is defined, and let $p$ be a generic $\lsig$-type of $(\EE, \FF)^\sharp$ over $M$. Let $e \models p$, and extend $p \cup (x \in eH_2)$ to a compete type $q$ over $Me$, generic in $eH_2$. By Lemma \ref{fininddense}, $q$ is very dense in $(\EE, \FF)^\sharp$ over $Me$ and $\deg_\sigma(q) = \RM(\EE) = \deg_\sigma(p)$. But $\sigma$-degree witnesses forking, and $q$ clearly forks over $M$, yielding the desired contradiction.
\end{proof}

We are finally ready to prove the key proposition.
 \begin{prop}
 \label{groupkey}

 If $H$ is an $\lsig$-definable subgroup of an $L$-definable group $\AAA$ with $\sigdeg(H)$ defined, then there are \begin{itemize}
 \item a finite-index subgroup $K \leq H$,
 \item a group correspondence $(\EE, \FF)^\sharp$,
 \item and an injective $\lsig$-definable group homomorphism $K \rightarrow (\EE, \FF)^\sharp$ whose image has finite index in $(\EE, \FF)^\sharp$.
\end{itemize}
\end{prop}

\begin{proof}
 Let $E$ be a small, somewhat saturater model of $T_A$ over which everything is defined.
 Let $p \in S(E)$ be a model-theoretically generic type in $H$ with $\sigdeg(p)$ defined.
 Then the $\sigma$-degree of all generic types of $H$ over $E$ is defined and equal to $\sigdeg(p)$, since those types are $H(E)$-translates of $p$. 

The statement of the theorem is invariant under prolongations, so we may assume without loss of generality (up to replacing $p$ and $H$ by their $(N-1)$st prolongations, and $\AAA$ by $\AAA \times \AAA^\sigma \times \ldots \AAA^{(\sigma^{N-1})}$ for $N$ from the definition of $\sigma$-degree) that $\acl(Ea) = \acl(E \sigma(a))$ for any realization $a$ of any generic type $q \in S(E)$ of $H$. This prolongation is the ``injective $\lsig$-definable group homomorphism'' from the conclusion of the theorem. We now look for $\EE$ and $\FF$.

We now apply lemma \ref{groupwrap} to $H$ inside $\AAA$ to get $H_1 \leq \EE$. So $H_1$, an $\lsig$-definable finite index subgroup of $H$, is \zdense in $\EE$, an irreducible, $L$-definable subgroup of $\AAA$.

 Then we apply lemma \ref{groupwrap} to the first prolongation $H_1^+$ of $H_1$ inside $\EE \times \EE^\sigma$ to get $\widetilde{H_2} \leq \FF$.
 So $\widetilde{H_2}$, an $\lsig$-definable finite index subgroup of $H_1^+$, is \zdense in $\FF$, an irreducible, $L$-definable subgroup of $\EE \times \EE^\sigma$.
  Let $K := \{ a \sthat \exists\, b\, (a,b) \in \widetilde{H_2} \}$, the projection of $\widetilde{H_2}$ into $\EE$. Since $\widetilde{H_2}$ is a subset of $H_1^+$, for every $(a, b) \in \widetilde{H_2}$ we have $b = \sigma(a)$; in other words, $\widetilde{H_2}$ is the first prolongation of $K$. Since $\widetilde{H_2} = K^+$ is a finite-index subgroup of $H_1^+$, it follows that $K$ is a finite-index subgroup of $H_1$.
 Now $H_1$ is \zdense in $\EE$, so by Corollary \ref{fininddense}, $K$ is \zdense in $\EE$.

 Let $q \in S(E)$ be a model-theoretically generic type in $K$. Since both inclusions $K \leq H_1 \leq H$ have finite index, $q$ is generic in $H$. So $\acl(Ea) = \acl(E \sigma(a))$ for any realization $a \models q$. Let $\phi(x,y)$ be the $L$-formula that witnesses this; that is, $\phi(x, \sigma(x)) \in q$ and $\acl(Eb) = \acl(Ec)$ for any $(b,c) \models \phi(x,y)$. Now the first prolongation $q^+$ is a generic type of $K^+$, therefore \zdense in $\FF$ by Lemma \ref{genericzdense}. Since $q^+$ contains the formula $\phi(x,y)$, this implies that $\RM(\FF \cap \phi(x,y)) = \RM(\FF)$. The purpose of $\phi$ is that the projection onto the first coordinate from $\FF \cap \phi(x,y)$ to $\EE$ has finite fibers, and the image of this projection contains $q$ which is \zdense in $\EE$. So this projection is a finite dominant function from $\FF \cap \phi(x,y)$ to $\EE$ , so $\RM(\FF \cap \phi(x,y)) = \RM(\EE)$. Thus, $\RM(\FF) = \RM(\EE)$ and the projection from $\FF$ to $\EE$ is finite dominant, so we may write $(\EE, \FF)^\sharp$.

 We already saw that $K$ is a finite index subgroup of $H$, that $K^+ \leq \FF$, and that $K$ is \zdense in $\EE$; now Lemma \ref{verydensegeneric} finishes the proof.

  \end{proof}

Without assuming that the $\sigma$-degree of $H$ is defined, it is easy enough to find a monogeny from $H$ to some $(\EE, \FF)^{sh}$ with the projections from $\FF$ to $\EE$ and $\EE^\sigma$ dominant but not necessarily finite. In ACFA, it is also possible to make the image of the isogeny have finite index in $(\EE, \FF)^{sh}$ by prolonging far enough. However, this relies on the Noetherianity of perfect difference ideals, and we do not expect the generalization to $T_A$ to be straightforward, even if possible.

\subsection{Characterizing groups and obtaining commutative diagrams.}

To remove the assumption that $H$ is a subgroup of an $L$-definable group, recall Fact \ref{amadorfact} (Theorem 2.15 in \cite{amador}).

\begin{fact*} (Fact \ref{amadorfact})
For any $\lsig$-definable group $H$, there are a finite-index subgroup $H'$ of $H$, an $L$-definable group $\AAA$, and an $\lsig$-definable group homomorphism $\phi: H' \rightarrow \AAA$ with finite kernel.
\end{fact*}

We use this fact to characterize $\lsig$-definable groups with defined $\sigma$-degree.

\begin{thm} \label{classifygp}
 For any $\lsig$-definable group $H$ with a generic type $p$ with defined $\sigdeg(p)$, there are
 \begin{itemize}
 \item a finite-index subgroup $K \leq H$,
 \item a group correspondence $(\EE, \FF)^\sharp$,
 \item and an $\lsig$-definable group homomorphism $\phi: K \rightarrow (\EE, \FF)^\sharp$ whose image has finite index in $(\EE, \FF)^\sharp$ and whose kernel is finite.
\end{itemize}
\end{thm}

\begin{proof}
 Let $E$ be a small model of $T_A$ over which everything is defined.
 First, we use Fact \ref{amadorfact} to obtain $H'$ and $\AAA$.
 Since $E$ is a model, $p$ specifies a coset of $H'$ and $E(H)$ contains an element $b$ of that coset. Let $p' := b^{-1}p$, an $H(E)$-translate of $p$ which is inside $H'$. Now $\sigdeg(p')$ is defined: the same as $\sigdeg(p)$, since by Corollary \ref{a3} $\sigma$-degree is invariant under $\lsig$-definable bijections such as $x \mapsto b^{-1}x$. Also, $p'$ is a generic of $H$, and therefore also a generic of $H'$. Thus, Proposition \ref{groupkey} can be applied to $p'$, $H'$, and $\AAA$. To finish the proof, observe that a composition of two \emph{$\lsig$-definable group homomorphisms with finite kernels and finite-index domains}, called \emph{monogenies} in \cite{amador}, is another monogeny.
\end{proof}

The next corollary connects the conclusions of the last theorem to the hypotheses of the diagram-producing Corollary \ref{moshediag}.

\begin{sheep}
With the hypotheses and notation of Theorem \ref{classifygp}, any generic type $p_1$ in $H$ is interalgebraic with some generic type $q_1$ in $(\EE, \FF)^\sharp$ whose realizations are very dense in $(\EE, \FF)^\sharp$.
\end{sheep}

\begin{proof}
 As in the proof of the theorem, let $p_1'$ be a translate of $p_1$ inside $K$, and let $q_1 := \phi(p_1')$. Since $p_1'$ is generic in $K$ and the kernel of $\phi$ is finite, it follows that $q_1$ is generic in $\phi(K)$. Since $\phi(K)$ has finite index in $(\EE, \FF)^\sharp$, it follows that $q_1$ is generic in $(\EE, \FF)^\sharp$. Lemma \ref{fininddense} 
 finishes the proof.
 \end{proof}

\begin{sheep}
\label{gpdiagta}
 If some very dense $\lsig$-type $r$ in $(\AAA, \BB)^\sharp$ is interalgebraic with some generic type $p$ of some $\lsig$-definable group $H$, then there is a group correspondence $(\EE,\FF)^\sharp$, and further irreducible sets $\CC$ and $\DD$ and $L$-definable finite dominant functions such that the following diagram commutes:
 \begin{equation*}\begin{CD}
\EE @<<< \FF @>>> \EE^{\sigma}\\
@AA{\rho}A @AAA @AA{\rho^\sigma}A\\
\CC @<<< \DD @>>> \CC^{\sigma}\\
@VV{\pi}V @VVV @VV{\pi^\sigma}V\\
\AAA @<<< \BB @>>> \AAA^{\sigma}\\
\end{CD}\end{equation*}

 In this diagram, the horizontal arrows are projections, since $\BB \subset \AAA \times \AAA^\sigma$, $\DD \subset \CC \times \CC^\sigma$, and $\FF \subset \EE \times \EE^\sigma$, and the two middle vertical arrows are restrictions of $\pi \times \pi^\sigma$ and $\rho \times \rho^\sigma$ to $\DD$.
\end{sheep}

\begin{proof}

 Any type in $(\AAA, \BB)^\sharp$ has defined $\sigma$-degree, so by Lemma \ref{a3} $p$ also has defined $\sigma$-degree. Thus we can apply the last corollary to $p$ and $H$ to obtain the group correspondence $(\EE,\FF)^\sharp$ and a very dense type $q_1$ in it, interalgebraic with $p$ and, therefore, also with $r$. Now Corollary \ref{moshediag} applies to $r$ in $(\AAA, \BB)^\sharp$ and $q_1$ in $(\EE, \FF)^\sharp$, providing the diagram.
\end{proof}

 We return briefly to ACFA to prove Corollary \ref{gpdiagthm}; we repeat the statement here.

\begin{sheep*} (Corollary \ref{gpdiagthm}) (ACFA)
 If some very dense $\lsig$-type $p$ in $(\AAA, \BB)^\sharp$ is grouplike
, then there is a group correspondence $(\EE,\FF)^\sharp$, and further irreducible sets $\CC$ and $\DD$ and $L$-definable finite dominant functions such that the following diagram commutes:
 \begin{equation*}\begin{CD}
\EE @<<< \FF @>>> \EE^{\sigma}\\
@AA{\rho}A @AAA @AA{\rho^\sigma}A\\
\CC @<<< \DD @>>> \CC^{\sigma}\\
@VV{\pi}V @VVV @VV{\pi^\sigma}V\\
\AAA @<<< \BB @>>> \AAA^{\sigma}\\
\end{CD}\end{equation*}

 In this diagram, the horizontal arrows are projections, since $\BB \subset \AAA \times \AAA^\sigma$, $\DD \subset \CC \times \CC^\sigma$, and $\FF \subset \EE \times \EE^\sigma$, and the two middle vertical arrows are restrictions of $\pi \times \pi^\sigma$ and $\rho \times \rho^\sigma$ to $\DD$.

\end{sheep*}

\begin{proof}
 The hypothesis `` p is grouplike'', i.e. minimal and non-orthogonal to a generic type of an $\lsig$-definable modular group, is clearly stronger that the hypothesis `` p is interalgebraic with a generic type of an $\lsig$-definable group'' in the last corollary.
\end{proof}

\section{Chasing diagrams}
\label{chase}

 In this section we make heavy use of the ideas in section \ref{funnylemmas} to chase the diagram obtained in Corollary \ref{gpdiagthm}, or, to be more precise, to count the degrees of functions in that diagram. Halfway through this section we restrict our attention to correspondences $(\AAA, \BB)^\sharp$ where $\BB$ is the graph of a function from $\AAA$ to $\AAA^\sigma$. We do not know how to remove this restriction from the algebraic geometry arguments in the second half of the author's thesis (to be exposed in another paper) toward which we are building in this paper. However, even the restricted result has been very useful in \cite{polydyn}.

It seems that there should be a slicker proof of Proposition \ref{firstchaseprop}
 that does not rely on the top row being a group correspondence, but I cannot make it work. 
 In particular, I have neither proof nor counterexample to the following refinement of a special case of Theorem \ref{moshethm}:

\begin{wish}
 If $(\AAA, \BB)^\sharp$ and $(\CC, \DD)^\sharp$ are definably isomorphic, then the degrees of the projections in the correspondence $(\EE, \FF)^\sharp$ obtained in Theorem \ref{moshethm} can be bounded by the degrees of the projections in $(\AAA, \BB)^\sharp$ and $(\CC, \DD)^\sharp$.
\end{wish}

 In the special case with a group correspondence in the top row of the diagram, we can make do without this wish, by means of a somewhat opaque diagram chase given below. It is given in rather more detail in the author's thesis \cite{mythesis}, in the language of algebraic geometry.

\subsection{One diagram chase.}

 The purpose of this section is Proposition \ref{firstchaseprop} which bounds the degrees of the functions in the middle row of the diagram in Corollary \ref{gpdiagthm} by the degrees of the functions in the bottom row.

\begin{prop}
\label{firstchaseprop}
 Given the commutative diagram of irreducible $L$-definable sets and finite dominant rational functions from Corollary \ref{gpdiagthm}

\begin{equation*}\begin{CD}
\EE @<\psi<< \FF @>\phi>> \EE^{\sigma}\\
@AA{\rho}A @AAA @AA{\rho^\sigma}A\\
\CC @<\beta<< \DD @>\alpha>> \CC^{\sigma}\\
@VV{\pi}V @VVV @VV{\pi^\sigma}V\\
\AAA @<g<< \BB @>f>> \AAA^{\sigma}\\
\end{CD}\end{equation*}

 where  \begin{itemize}
\item  the horizontal arrows are projections, so $\BB \subset \AAA \times \AAA^\sigma$, $\DD \subset \CC \times \CC^\sigma$, and $\FF \subset \EE \times \EE^\sigma$;
\item $(\EE, \FF)^\sharp$ is a group correspondence; and
\item the two middle vertical arrows are restrictions of $\pi \times \pi^\sigma$ and $\rho \times \rho^\sigma$ to $\DD$;
\end{itemize}
 we construct another diagram of the same shape, with the same properties and the same $\AAA$, $\BB$, $f$ and $g$, satisfying an additional assumption that $\deg(\beta) \leq \deg(g)$.
\end{prop}

The rest of this section is the proof of the proposition. If the original diagram already satisfies the additional assumption, we are done. Otherwise, we construct another diagram of the same shape with a lower-degree $\pi$, and induct on $\deg(\pi)$. Note that when $\deg(\pi) = 1$, the additional assumption is automatically true, so this induction has a base case. We begin the induction step by finding a non-trivial shared initial factor of $\alpha$ and $\pi \circ \beta$.

\begin{llama}
 In the diagram above, if $\deg(\beta) > \deg(g)$ then there is a non-trivial shared initial factor of $\alpha$ and $\pi \circ \beta$.
\end{llama}

\begin{proof}
 Consider the $\deg(\pi) \deg(\beta)$ points in a generic fiber $F := (\pi \circ \beta)^{-1} (a)$. If $\alpha$ and $\pi \circ \beta$ do not share any nontrivial initial factors, $\alpha(F)$ has the same size as $F$. By the commutativity of the diagram,
 $\alpha ( (\pi \circ \beta)^{-1} (a) ) = (\pi^\sigma)^{-1} (f (g^{-1} (a)))$ which has at most $\deg(\pi^\sigma) \deg(g)$ many points, which is not enough if $\deg(\beta) > \deg(g)$.
\end{proof}

\begin{lett}
 Let $\eta$ be the (nontrivial according to the last lemma) maximal shared initial factor of $\alpha$ and $\pi \circ \beta$ given by Lemma \ref{shareinit}.\\
 Let $\lambda$ be the least common extension of $\eta$ and $\beta$, provided by Lemma \ref{leastclosing} as they have a common extension $\pi \circ \beta$. \\
 Let $\CC_1$ be the image of $\lambda$, and let $\pi_1 : \CC_1 \rightarrow \AAA$ and $\pi_2 : \CC \rightarrow \CC_1$ be such that
 $ \pi_1 \circ \pi_2 = \pi$ and $\pi_2 \circ \beta = \lambda$.
\end{lett}

 Since $\eta$ and $\beta$ share an extension $\pi \circ \beta$, Lemma \ref{leastclosing} produces $\lambda$. Since $\alpha$ and $\beta$ do not share nontrivial initial factors ($\DD$ is a \emph{subset} of $\CC \times \CC^\sigma$), it follows that $\eta$ is not an initial factor of $\beta$, so $\lambda$ is a proper extension of $\beta$, i.e. $\deg(\pi_2) \neq 1$ and so $\deg(\pi_1) \lneq \deg(\pi)$.

 Now we use Lemma \ref{groupclosing} to close the group correspondence in the top row:
\begin{lett}
 Let $\theta$ be a group homomorphism which is a common extension of $\phi$ and $\psi$.\\
 Let $\EE_1$ be the image of $\theta$.\\
 Let $\rho_2$ be the restriction of $\rho \times \rho^\sigma$ to $\DD$.
 Let $\zeta := \theta \circ \rho_2: \DD \rightarrow \EE_1$.
\end{lett}

 Tracing $\zeta$ along the left side of the diagram, we see that it factors through $\beta$. Tracing it along the right side, we see that it factors through $\alpha$ and, therefore, $\eta$. Therefore, the least common extension $\lambda$ of $\beta$ and $\eta$ is an initial factor of $\zeta$.

\begin{lett}
 Let $\rho_1 : \CC_1 \rightarrow \EE_1$ be such that $\zeta = \rho_1 \circ \lambda$.
\end{lett}

 Now $\CC_1$, $\EE_1$, $\pi_1$, and $\rho_1$ constitute the left column of the new diagram. Applying $\sigma$ to them, we obtain the right column of the new diagram. To finish, we define

\begin{lett}
 Let $\DD_1 := (\pi_2 \times \pi_2^\sigma) (D)$ and let $\FF_1 := (\rho_1 \times \rho_1^\sigma) (\DD_1)$.
\end{lett}

 The new diagram clearly commutes. We only need to show that $(\EE_1, \FF_1)^\sharp$ is a group correspondence.

\begin{llama}
 $\FF_1$ is a subgroup of $\EE_1 \times \EE_1^\sigma$.
\end{llama}

\begin{proof}
 It is sufficient to show that there is a group homomorphism $\gamma: \EE \rightarrow \EE_1$ such that $\gamma \circ \rho = \rho_1 \circ \pi_2$, because then $\FF_1 = (\gamma \times \gamma^\sigma) (\FF)$, an image of a subgroup under a group homomorphism.
 Since $\psi$ is an initial factor of $\theta$, let $\gamma$ be such that $\theta = \gamma \circ \psi$.
 Now $\zeta := \theta \circ \rho_2 = \gamma \circ \psi \circ \rho_2 = \gamma \circ \rho \circ \beta$.
 But also $\zeta = \rho_1 \circ \lambda$ and $\lambda := \pi_2 \circ \beta$, so  $\zeta = \rho_1 \circ \pi_2 \circ \beta$. So
 $$\gamma \circ \rho \circ \beta = \zeta = \rho_1 \circ \pi_2 \circ \beta$$
Since $\beta$ is surjective, it can be canceled to give
 $$\gamma \circ \rho = \rho_1 \circ \pi_2$$

The lemma is now proved, the induction step of the proof of the proposition is competed, and we are done.
\end{proof}

\subsection{Another diagram chase.}

 We now throw up our hands and give up on correspondences; we
 restrict our attention to the case where $\BB$ is the graph of a function $f: \AAA \rightarrow \AAA^\sigma$. We have no idea how to get around this restriction, which is most vexing. For this special case, Proposition \ref{firstchaseprop} becomes

 \begin{sheep}
\label{firstchasepropfuncinside}
 Given the commutative diagram of irreducible $L$-definable sets and finite dominant rational functions from Theorem \ref{gpdiagthm}

\begin{equation*}\begin{CD}
\EE @<\psi<< \FF @>\phi>> \EE^{\sigma}\\
@AA{\rho}A @AAA @AA{\rho^\sigma}A\\
\CC @<\beta<< \DD @>\alpha>> \CC^{\sigma}\\
@VV{\pi}V @VVV @VV{\pi^\sigma}V\\
\AAA @<\id<< \BB @>f>> \AAA^{\sigma}\\
\end{CD}\end{equation*}

 where \begin{itemize}
 \item the horizontal arrows are projections, so $\BB \subset \AAA \times \AAA^\sigma$, $\DD \subset \CC \times \CC^\sigma$, and $\FF \subset \EE \times \EE^\sigma$;
 \item $(\EE, \FF)^\sharp$ is a group correspondence;
 \item and the two middle vertical arrows are restrictions of $\pi \times \pi^\sigma$ and $\rho \times \rho^\sigma$ to $\DD$;
\end{itemize}
 we construct another diagram of the same shape, with the same properties and the same $\AAA$ and $f$, satisfying an additional assumption that $\beta = \id$.
\end{sheep}

\begin{proof}
 This is precisely Proposition \ref{firstchaseprop} with $g = \id$; the $\deg(\beta) \leq \deg(g)$ implies that $\beta$ is a bijection that can be absorbed into $\alpha$.
\end{proof}

 The purpose of this section is to turn the top row of the diagram in Corollary \ref{firstchasepropfuncinside} into a function as well. The following lemma provides an induction step for the induction on $\deg(\rho)$; the base case $\deg(\rho)=1$ is clear.

\begin{llama}
 If the $\deg(\psi) \gneq 1$ in the top half
 \begin{equation*}\begin{CD}
\EE @<<\psi< \FF @>\phi>> \EE^{\sigma}\\
@AA{\rho}A @A{r}AA @AA{\rho^\sigma}A\\
\CC @<<{\id}< \CC @>{\alpha}>> \CC^{\sigma}\\
\end{CD}\end{equation*}
of the diagram in the last corollary, then there is another diagram with the same properties, with the same $\CC$ and $\alpha$, and with a lower-degree $\rho$.
\end{llama}

\begin{proof}
 We put the fact that $\rho = \psi \circ r$ into the diagram and let
 $$\AAA := (r \times r^\sigma) (\mbox{the graph of $\alpha$})$$
 This $\AAA$ is entirely unrelated to the $\AAA$ in the rest of the paper. Note that $\AAA$ is irreducible, being the image of an irreducible graph of $\alpha$ (isomorphic to the irreducible $\CC$) under a finite map $(r \times r^\sigma)$. So we get

\begin{equation*}\begin{CD}
\EE @<<\psi< \FF @>\phi>> \EE^{\sigma}\\
@AA{\psi}A @AAA @AA{\psi^\sigma}A\\
\FF @<<< \AAA @>>> \FF^\sigma\\
@AA{r}A @AAA @AA{r^\sigma}A\\
\CC @<<{\id}< \CC @>{\alpha}>> \CC^{\sigma}\\
\end{CD}\end{equation*}

 Let $\BB := (\psi \times \psi^\sigma)^{-1} (\FF)$, a subgroup of $\FF \times \FF^\sigma$, and note that $\AAA \subset \BB$. It is possible that $\dM(\BB) \gneq 1$; let $\BB_0$ be the connected component of (the $\omega$-stable group) $\BB$.
 Since $\AAA$ and $\BB$ have the same Morley rank and $\AAA$ is irreducible, $\AAA$ must be \zdense in some coset $\BB_1$ of $\BB_0$.
 It follows from the fact that models of $T_A$ are existentially closed that translation by an appropriate element of $\FF$ twists $\AAA$ to (being \zdense in) $\BB_0$, finishing the proof.
\end{proof}

We have accomplished the purpose of this section:

 \begin{sheep}
\label{firstchasepropfunc}
 Given the commutative diagram of irreducible $L$-definable sets and finite dominant rational functions from Corollary \ref{gpdiagthm}

\begin{equation*}\begin{CD}
\EE @<\psi<< \FF @>\phi>> \EE^{\sigma}\\
@AA{\rho}A @AAA @AA{\rho^\sigma}A\\
\CC @<\beta<< \DD @>\alpha>> \CC^{\sigma}\\
@VV{\pi}V @VVV @VV{\pi^\sigma}V\\
\AAA @<\id<< \AA @>f>> \AAA^{\sigma}\\
\end{CD}\end{equation*}

 where  \begin{itemize}
 \item the horizontal arrows are projections, so $\BB \subset \AAA \times \AAA^\sigma$, $\DD \subset \CC \times \CC^\sigma$, and $\FF \subset \EE \times \EE^\sigma$;
 \item $(\EE, \FF)^\sharp$ is a group correspondence;
 \item and the two middle vertical arrows are restrictions of $\pi \times \pi^\sigma$ and $\rho \times \rho^\sigma$ to $\DD$;
\end{itemize}
 we construct the following diagram with the same $\AAA$ and $f$

\begin{equation*}\begin{CD}
\EE @>\phi>> \EE^{\sigma}\\
@AA{\rho}A @AA{\rho^\sigma}A\\
\CC @>{\alpha}>> \CC^{\sigma}\\
@VV{\pi}V@VV{\pi^\sigma}V\\
\AAA @>f>> \AAA^{\sigma}\\
\end{CD}\end{equation*}

 Where all sets are irreducible, all functions are finite dominant rational, and $\phi: \EE \rightarrow \EE^\sigma$ is a group homomorphism.
\end{sheep}

\subsection{What were we chasing after?}

 Here we combine the technical results of the last two sections with Corollaries \ref{gpdiagthm} for ACFA and \ref{gpdiagta} for a general $T_A$.

\begin{sheep}
 If some very dense $\lsig$-type $p$ in $(\AAA, \BB)^\sharp$ is interalgebraic with some generic type of some $\lsig$-definable group of finite $U$-rank, then there is a group correspondence $(\EE,\FF)^\sharp$, and further irreducible sets $\CC$ and $\DD$ and $L$-definable finite dominant functions such that the following diagram commutes:
 \begin{equation*}\begin{CD}
\EE @<<< \FF @>>> \EE^{\sigma}\\
@AA{\rho}A @AAA @AA{\rho^\sigma}A\\
\CC @<<< \DD @>>> \CC^{\sigma}\\
@VV{\pi}V @VVV @VV{\pi^\sigma}V\\
\AAA @<<< \BB @>>> \AAA^{\sigma}\\
\end{CD}\end{equation*}

 In this diagram, the horizontal arrows are projections, since $\BB \subset \AAA \times \AAA^\sigma$, $\DD \subset \CC \times \CC^\sigma$, and $\FF \subset \EE \times \EE^\sigma$, and the two middle vertical arrows are restrictions of $\pi \times \pi^\sigma$ and $\rho \times \rho^\sigma$ to $\DD$.
 Further, $\deg(\DD \rightarrow \CC) \leq \deg(\BB \rightarrow \AAA)$
 \end{sheep}
\begin{proof}
Corollary \ref{gpdiagta} and Proposition \ref{firstchaseprop}.
\end{proof}

\begin{sheep}
\label{groupfuncdiagta}
 If some very dense type in $(\AAA, f)^\sharp$ is interalgebraic with a generic type of some $\lsig$-definable group, then there are
  an irreducible group $\EE$ and a finite dominant rational group homomorphism $\phi: \EE \rightarrow \EE^\sigma$,
  an irreducible $\CC$ and a finite dominant rational $\alpha: \CC \rightarrow \CC^\sigma$,
  and $L$-definable finite dominant rational $\pi$ and $\rho$ such that the following diagram commutes:

\begin{equation*}\begin{CD}
\EE @>\phi>> \EE^{\sigma}\\
@AA{\rho}A @AA{\rho^\sigma}A\\
\CC @>{\alpha}>> \CC^{\sigma}\\
@VV{\pi}V@VV{\pi^\sigma}V\\
\AAA @>f>> \AAA^{\sigma}\\
\end{CD}\end{equation*}
\end{sheep}

\begin{proof}
Corollary \ref{gpdiagta} and \ref{firstchasepropfunc}.
\end{proof}

 In ACFA, the above applies when some minimal grouplike $\lsig$-type is very dense in $(\AAA, \BB)^\sharp$.

 \begin{thm} (ACFA)
\label{groupfuncdiagthm}
 Given an $L$-definable $\AAA$ of Morley degree $1$ and an $L$-definable finite (and, therefore, dominant) rational $f: \AAA \rightarrow \AAA^\sigma$, suppose that some very dense type in $(\AAA, f)^\sharp$ is grouplike. Then there are
  an irreducible group $\EE$ and a finite dominant rational group homomorphism $\phi: \EE \rightarrow \EE^\sigma$,
  an irreducible $\CC$ and a finite dominant rational $\alpha: \CC \rightarrow \CC^\sigma$,
  and $L$-definable finite dominant rational $\pi$ and $\rho$ such that the following diagram commutes:

\begin{equation*}\begin{CD}
\EE @>\phi>> \EE^{\sigma}\\
@AA{\rho}A @AA{\rho^\sigma}A\\
\CC @>{\alpha}>> \CC^{\sigma}\\
@VV{\pi}V@VV{\pi^\sigma}V\\
\AAA @>f>> \AAA^{\sigma}\\
\end{CD}\end{equation*}
\end{thm}

\section{Characterization for curves in ACFA}
\label{forwref}

 Corollary \ref{groupfuncdiagta} is the most we can prove for $T_A$ for an arbitrary $T$.
 Restricting to ACFA and to $\sigma$-degree $1$ allows us to obtain a much stronger result, relying on jet spaces of varieties and on Hurwitz-Riemann equations for ramification loci of rational functions between curves. Before we can do that, however, we must translate the conclusion of the theorem into the language of algebraic geometry, which is not exactly the same as model theory of algebraically closed fields. In particular, we now pay for having redefined the notion of ``rational function''; to distinguish them, we will write \emph{rational morphism} for the notion from algebraic geometry. We end this paper with that translation. The algebraic geometry results that are needed for the final Theorem \ref{endthm} belong to algebraic geometry and may be of independent interest, so they get a paper of their own \cite{secondpaper}.

  here are the main ideas of the translation: \begin{enumerate}
  \item If $(K, \sigma)$ is a model of ACFA and $\Phi$ is the Frobenius automorphism on $K$, then $(K, \sigma \circ \Phi^n)$ is also a model of ACFA, for any $n \in \mathbb{Z}$ (Corollary 1.12 in \cite{ChaHru1}).
  \item By quantifier elimination in ACF, definable sets are constructible: up to subsets of lower Morley rank they are (classical affine) varieties.
  \item Rational functions on curves are (equivalent to functions) of the form $f \circ \Phi^n$ where $f$ is a separable rational morphism, $\Phi$ is the Frobenius automorphism, and $n \in \mathbb{Z}$.
    \item Purely inseparable functions $f$ are known to give rise to \emph{field-like} $(\AAA, f)^\sharp$ (\cite{ChaHru1}), so we may exclude them from our consideration. The rest are known not be fieldlike, so the diagram in Theorem \ref{groupfuncdiagthm} is not only necessary but also sufficient for $(\AAA, f)^\sharp$ to be grouplike (Theorem 4.5 in \cite{ChaHru1}, but see also example 6.6 in \cite{ChaHru1}).
  \end{enumerate}

The last two items are false for varieties of higher dimension.

   With these ideas we attack the diagram in the conclusion of Theorem \ref{groupfuncdiagthm}:
\begin{equation*}\begin{CD}
\EE @>\phi>> \EE^{\sigma}\\
@AA{\rho}A @AA{\rho^\sigma}A\\
\CC @>{\alpha}>> \CC^{\sigma}\\
@VV{\pi}V@VV{\pi^\sigma}V\\
\AAA @>f>> \AAA^{\sigma}\\
\end{CD}\end{equation*}

 First, by taking Zariski closures and trimming lower-dimensional components (Morley degree is insensitive to them, so they may have snuck in), we may assume that $\AAA$, $\CC$, and $\EE$ are varieties.
 To do anything about the arrows, we must restrict to $\sigma$-degree $1$, that is we must require $\AAA$, $\CC$, and $\EE$ to be curves.
 Writing $\pi =: \pi' \circ \Phi^n$ and replacing the top half of the diagram by $\Phi^n(\mbox{ the original top half })$, we may assume that $\pi$ is a separable rational morphism. Writing $\rho =: \Phi^m \circ \rho'$ and replacing the top row by $\Phi^{n-m}(\mbox{ original top row })$, we may assume that $\rho$ is a separable rational morphism.
 Counting inseparable degrees, we may write $f =: \Phi^k \circ f^\tau$, $\alpha =: \Phi^k \circ \alpha^\tau$, and $\phi =: \Phi^k \circ \phi^\tau$ \emph{all with the same $k$}, where $\tau := \sigma \circ \Phi^{-k}$, turning the diagram into
\begin{equation*}\begin{CD}
\EE @>\phi^\tau>> \EE^\tau @>\Phi^k>> \EE^{\sigma}\\
@AA{\rho}A @AA{\rho^\tau}A @AA{\rho^\sigma}A\\
\CC @>{\alpha^\tau}>> \CC^\tau @>\Phi^k>> \CC^{\sigma}\\
@VV{\pi}V @VV{\pi^\tau}V @VV{\pi^\sigma}V\\
\AAA @>f^\tau>> \AAA^\tau @>\Phi^k>> \AAA^{\sigma}\\
\end{CD}\end{equation*}
 where all arrows in the left half of the diagram are separable rational morphisms. We apply the algebraic geometry theorems to the left half of the diagram characterizing $f^\tau$, and then add $\Phi^k$ into the characterization to describe $f$.

  The next two theorems, the first easy and the second not so easy, are proved in the author's thesis \cite{mythesis} and will appear in another paper \cite{secondpaper}.

\begin{thm} \cite{secondpaper}
 If $(K, \tau)$ is a model of ACFA, $\EE$ is an algebraic group curve, $\phi': \EE \rightarrow \EE^\tau$ is an algebraic group homo- but not isomorphism and the following diagram of curves and finite separable rational morphisms commutes
 \begin{equation*}\begin{CD}
\EE @>\phi'>> \EE^\tau \\
@AA{\rho}A @AA{\rho^\tau}A \\
\CC @>{\alpha'}>> \CC^\tau \\
\end{CD}\end{equation*}
  Then there is a birational isomorphism $g: \CC \rightarrow \DD$ to an algebraic group $\DD$ and an isogeny $\psi: \DD \rightarrow \DD^{\tau}$ such that $\alpha' = g^{-1} \psi \circ g$.
\end{thm}

\begin{thm}\cite{secondpaper}
 If $(K, \tau)$ is a model of ACFA, $\DD$ is an algebraic group curve, $\psi: \DD \rightarrow \DD^\tau$ is an algebraic group homo- but not isomorphism, and the following diagram of curves and finite separable rational morphisms commutes
\begin{equation*}\begin{CD}
\DD @>\psi>> \DD^\tau \\
@VV{\pi}V @VV{\pi^\tau}V \\
\AAA @>f'>> \AAA^\tau \\
\end{CD}\end{equation*}
 Then there is another algebraic group $\tilde{\DD}$, an isogeny $\tilde{\psi} : \tilde{\DD} \rightarrow \tilde{\DD}^\sigma$, and a finite separable $\tilde{\pi}$ such that
\begin{equation*}\begin{CD}
\tilde{\DD} @>{\tilde{\psi}}>> {\tilde{\DD}^\tau} \\
@VV{\tilde{\pi}}V @VV{\tilde{\pi}^\tau}V \\
\AAA @>f'>> \AAA^\tau \\
\end{CD}\end{equation*}
 also commutes, and $\tilde{\pi}$ is the quotient of $\tilde{\DD}$ by a (finite) group of algebraic group automorphisms.
\end{thm}

 Since the (separable) degree of $f'$ is at least $2$, the genus of (the normalization of) $\AAA$ is at most $1$. Since any definable function between elliptic curves is an isogeny composed on translation, it is easy to see (see \cite{secondpaper}) that when $\AAA$ has genus $1$, any $f'$ of degree at least $2$ gives rise to a grouplike $(\AAA, f)^\sharp$. Thus, the interesting case is when $\AAA$ is (birational to) $\mathbb{A}^1$. Such rational functions $f'$ have all kinds of remarkable properties and have been studied since the nineteenth century. When $\DD$ is the multiplicative group, its only algebraic group automorphism of finite order is $x \mapsto \frac{1}{x}$, and the corresponding $f'$ are the Chebyshev polynomials.
  When $\DD$ is an elliptic curve, the degree of $\pi$ is small (generically, $2$, and at most $24$ for the worst positive characteristic, complex multiplication case \cite{silverell}) and $f'$ is called a Latt\`{e}s function; it is never a polynomial. In positive characteristic, $\DD$ may also be the additive group ($\tilde{\psi}$ is then a separable additive polynomial such as $x^p+x$).
  In that case, the degree of $\tilde{\pi}$ is bounded by the degree of $\tilde{\psi}$ (see \cite{secondpaper}). 
 Surely someone has called such $f'$ \emph{additive Latt\`{e}s functions}, and that is our terminology.
 Which all adds up to the following theorem that was so useful in \cite{polydyn}.

\begin{thm}\label{endthm}
 If $(K, \sigma)$ is a model of ACFA, $\AAA$ is an algebraic curve and \\ $f: \AAA \rightarrow \AAA^\sigma$ is a definable finite rational function, the following are equivalent:\begin{itemize}
 \item some type in $(\AAA, f)^\sharp$ is grouplike
 \item all types in $(\AAA, f)^\sharp$ are grouplike
 \item there is a birational morphism $h: \AAA \rightarrow \BB$ and a separable $f'$ with $\deg(f') \geq 2$ such that $f = (h^\sigma)^{-1} \circ \Phi^k \circ f' \circ h$ for some $k \in \mathbb{Z}$ and either $\BB$ is an elliptic curve, or $\BB = \mathbb{A}^1$ and $f'$ is a Chebyshev polynomial, a Latt\`{e}s function, or an additive Latt\`{e}s function.
 \end{itemize}
  Furthermore, this property of the pair $(\AAA, f)$ is first-order definable.
\end{thm}

\bibliography{bibforfirstpaper}{}
\bibliographystyle{plain}
\end{document}